\documentclass{article}

\usepackage[final]{neurips_2020}

\usepackage[utf8]{inputenc} % allow utf-8 input
\usepackage[T1]{fontenc}    % use 8-bit T1 fonts
\usepackage{hyperref}       % hyperlinks
\usepackage{url}            % simple URL typesetting
\usepackage{booktabs}       % professional-quality tables
\usepackage{amsfonts}       % blackboard math symbols
\usepackage{nicefrac}       % compact symbols for 1/2, etc.
\usepackage{microtype}
\usepackage{soul}     % microtypography
\usepackage[vlined,boxed,ruled]{algorithm2e}

\SetCommentSty{mycommfont}
\usepackage{algorithmic}
\usepackage{amsmath,epsfig,amsfonts,amssymb}
\usepackage{graphics,psfrag,theorem,calc,subfigure,url,bm,color}
\usepackage{amssymb}
\usepackage{amsmath}
\usepackage{latexsym}
\usepackage{mathrsfs}
\usepackage{enumerate}
\usepackage{bm}
\usepackage{color}
\usepackage{verbatim}
\usepackage{url}
\usepackage{graphicx}  % Written by David Carlisle and Sebastian Rahtz
\usepackage{stfloats}  % Written by Sigitas Tolusis
\usepackage{multirow}
\usepackage{threeparttable}
\usepackage{tcolorbox}
\usepackage{tikz}
\usepackage{cancel}
\usepackage{hyperref}
\usepackage{tabularx}
\usepackage{array}
\usepackage{float}
\hypersetup{
	colorlinks=true,
	linkcolor=blue,
	filecolor=magenta,
	urlcolor=cyan,
}

\newenvironment{proof}{{\bf Proof\,\,}}{\endproof\par}

\newcounter{spb}
\def \openbox{$\sqcup\llap{$\sqcap$}$}
\def \endproof{\enskip \null \nobreak \hfill \openbox \par}

\newtheorem{definition}{Definition}
\newtheorem{lemma}{Lemma}

\DeclareMathOperator{\proj}{proj}
\DeclareMathOperator{\dist}{dist}
\DeclareMathOperator{\PPA}{PPA\_sub}
\DeclareMathOperator{\ri}{ri}

\newcounter{dummy}
\numberwithin{dummy}{section}
\newtheorem{theorem}[dummy]{Theorem}

\newtheorem{proposition}[dummy]{Proposition}
\newtheorem{Remark}[dummy]{Remark}
\newtheorem{assumption}[dummy]{Assumption}

\newtheorem{example}[dummy]{Example}

\title{\Large Fast Epigraphical Projection-based Incremental Algorithms for \\
	\vspace{0.1\baselineskip} Wasserstein Distributionally Robust Support Vector Machine}

\author{
	Jiajin Li \\
Department of Systems Engineering and Engineering Management\\
	The Chinese University of Hong Kong \\
	\texttt{jjli@se.cuhk.edu.hk} \\
	\And
	Caihua Chen\thanks{Corresponding author} \\
	School of Management and Engineering\\
	 Nanjing University \\
	\texttt{chchen@nju.edu.cn} \\
	\And
	Anthony Man-Cho So \\
Department of Systems Engineering and Engineering Management\\
	The Chinese University of Hong Kong \\
	\texttt{manchoso@se.cuhk.edu.hk} 
}

\begin{document}
	
	\maketitle
	
	\begin{abstract}
		Wasserstein \textbf{D}istributionally \textbf{R}obust \textbf{O}ptimization (DRO) is concerned with finding decisions that
		perform well on data that are drawn from the worst-case probability distribution within a Wasserstein ball centered at a certain nominal
		distribution. In recent years, it has been shown that various DRO formulations of learning models admit tractable convex reformulations.
		However, most existing works propose to solve these convex reformulations by general-purpose solvers, which are not well-suited
		for tackling large-scale problems. In this paper, we focus on a family of Wasserstein distributionally robust support vector machine (DRSVM)
		problems and propose two novel epigraphical projection-based incremental algorithms to solve them. The updates in each iteration of
		these algorithms can be computed in a highly efficient manner. Moreover, we show that the DRSVM problems considered in this paper
		satisfy a Hölderian growth condition with explicitly determined growth exponents. Consequently, we are able to establish the convergence
		rates of the proposed incremental algorithms. Our numerical results indicate that the proposed methods are orders of magnitude faster than
		the state-of-the-art, and the performance gap grows considerably as the problem size increases.	
		
		%Wasserstein  \textbf{D}istributionally \textbf{R}obust \textbf{O}ptimization (DRO) seeks data-driven decisions that
		%perform well under the most adverse worst-case distribution within a Wasserstein ball centered at the empirical distribution of the training examples.
		%Unfortunately, these (semi-infinite) min-max formulations are, however, hard to address in general. Instead, various DRO formulations of learning models admit tractable convex dual forms and further can be tackled by the off-the-shelf solvers. Nevertheless, general-purpose solvers suffer from a substantial computational burden.
		%By fully exploiting the problem-specific structure, we propose two novel epigraphical projection-based incremental algorithms  to efficiently solve a host of  Wasserstein distributionally robust support vector machine (DRSVM) problems.
		%Specifically, our approach develops fast solutions to tackle the epigraphical projection and single sample proximal point update in linear time.
		%On the theoretical side,  various step size schemes and provable convergence rates have been established via exploring the Hölderian growth condition of the problem under different scenarios.
		%Finally, the experimental results indicate that the proposed method is orders of magnitude faster than the state-of-the-art, and the performance gap grows considerably with problem size.
	\end{abstract}

%	\vspace{-0.5\baselineskip}
	\section{Introduction}
%	\vspace{-0.5\baselineskip}
	%% (1) DRO relevance, why important? - model generalization/ regularization
	Wasserstein distance-based distributionally robust optimization (DRO) has recently received significant attention in the machine learning community. This can be attributed to its ability to improve generalization performance by robustifying the learning model against unseen data~\cite{lee2018minimax,shafieezadeh2019regularization}. The DRO approach offers a principled way to regularize empiricial risk minimization problems and provides a transparent probabilistic interpretation of a wide range of existing regularization techniques;
	%Arise from the DRO perspective, a flurry of literature in effect provides a probabilistic interpretation of existing regularization techniques and further offers alternative risk minimization models in a principled way;
	see, e.g.,~\cite{blanchet2016robust,gao2017Wasserstein,shafieezadeh2019regularization} and the references therein. Moreover, many representative distributionally robust learning models
	admit equivalent reformulations as tractable convex programs via strong duality~\cite{shafieezadeh2019regularization,lee2015distributionally,luo2017decomposition,wiesemann2014distributionally,kuhn2019wasserstein}.
	% solver limitation - motivated to design a fast iterative algo
	Currently, a standard approach to solving these reformulations is to use off-the-shelf solvers such as YALMIP or CPLEX. However, these  general-purpose solvers do not scale well with the problem size. Such a state of affairs greatly limits the use of the DRO methodology in machine learning applications and naturally motivates the study of the algorithmic aspects of DRO.
	
	% (2) DRSVM - relevance
	In this paper, we aim to design fast iterative methods for solving a family of Wasserstein distributionally robust support vector machine (DRSVM) problems. The SVM is one of the most frequently used classification methods and has enjoyed notable empirical successes in machine learning and data analysis~\cite{suykens1999least,singla2019survey}. However, even for this seemingly simple learning model, there are very few works addressing the development of fast algorithms for its Wasserstein DRO formulation, which takes the form $\inf\limits_{w} \{ \frac{c}{2}\|w\|_2^2+ \sup\limits_{\mathbb{Q} \in {B}_{\epsilon}^p(\hat{\mathbb{P}}_n)} \mathbb{E}_{(x,y)\sim\mathbb{Q}}[\ell_w(x,y)] \}$ and can be reformulated as
	% (3) probelm set up
	\begin{equation}
	\label{eq:drsvm_our}
	\begin{aligned}
	& \min_{w,\lambda} \lambda\epsilon + \frac{1}{n}\sum\limits_{i=1}^n \max \left\{1-w^Tz_i, 1+w^Tz_i-\lambda\kappa,0\right\} + \frac{c}{2}\|w\|_2^2,~\text{s.t.} \ \|w\|_q \leq \lambda;
	\end{aligned}
	\end{equation}
	see~\cite[Theorem 2]{lee2015distributionally} and~\cite[Theorem 3.11]{shafieezadeh2019regularization}. Problem~\eqref{eq:drsvm_our} arises from the vanilla soft-margin SVM model. Here, $\frac{c}{2}\|w\|_2^2$ is the regularization term with $c\ge0$; $x\in\mathbb{R}^d$ denotes a feature vector and $y\in\{-1,+1\}$ is the associated binary label; $\ell_w(x,y)=  \max\{ 1-y w^T x,0 \}$ is the hinge loss w.r.t.~the feature-label pair $(x,y)$ and learning parameter $w \in\mathbb{R}^d$; $\{(\hat{x}_i,\hat{y}_i)\}_{i=1}^n$ are $n$ training samples independently  and identically drawn from an unknown distribution $\mathbb{P}^*$ on the feature-label space $\mathcal{Z} = \mathbb{R}^d \times \{+1,-1\}$ and $z_i = \hat{x}_i \odot \hat{y}_i$; $\hat{\mathbb{P}}_n =\tfrac{1}{n} \sum_{i=1}^n \delta_{(\hat{x}_i,\hat{y}_i)}$ is the empirical distribution associated with the training samples; ${B}_{\epsilon}^p(\hat{\mathbb{P}}_n) = \{ \mathbb{Q} \in \mathcal{P}(\mathcal{Z}): W_p(\mathbb{Q},\hat{\mathbb{P}}_n) \le \epsilon \} $ is the ambiguity set defined on the space of probability distributions $\mathcal{P}(\mathcal{Z})$ centered at the empirical distribution $\hat{\mathbb{P}}_n$ and has radius $\epsilon\ge0$ w.r.t. the $\ell_p$ norm-induced Wasserstein distance
	\[ W_p(\mathbb{Q},\hat{\mathbb{P}}_n) = \inf_{\Pi\in\mathcal{P}(\mathcal{Z}\times\mathcal{Z})} \left\{ \int_{\mathcal{Z}\times\mathcal{Z}} d_p(\xi,\xi')~\Pi({\rm d}\xi,{\rm d}\xi') : \Pi({\rm d}\xi,\mathcal{Z})=\mathbb{Q}({\rm d}\xi),\, \Pi(\mathcal{Z},{\rm d}\xi') = \hat{\mathbb{P}}_n({\rm d}\xi') \right\}, \]
	where $\xi=(x,y)\in\mathcal{Z}$, $\tfrac{1}{p}+\tfrac{1}{q} =1$, and $d_p(\xi,\xi') = \|x-x'\|_p+\tfrac{\kappa}{2}|y-y'|$ is the transport cost between two data points $\xi,\xi'\in\mathcal{Z}$ with $\kappa \ge 0$ representing the relative emphasis between feature mismatch and label uncertainty. In particular, the larger the $\kappa$, the more reliable are the labels; see~\cite{shafieezadeh2019regularization,li2019first} for further details. Intuitively, if the ambiguity set ${B}_{\epsilon}^p(\hat{\mathbb{P}}_n)$ contains the ground-truth distribution $\mathbb{P}^*$, then the estimator $w^*$ obtained from an optimal solution to \eqref{eq:drsvm_our} will be less sensitive to unseen feature-label pairs.
	
	% (4): limitations for previous work : model(SGD), algorithm(DRLR)
	In the works~\cite{lee2015distributionally,luo2017decomposition}, the authors proposed cutting surface-based methods to solve the $\ell_p$-DRSVM problem~\eqref{eq:drsvm_our}. However, in their implementation, they still need to invoke off-the-shelf solvers for certain tasks. Recently, researchers have proposed to use stochastic (sub)gradient descent to tackle a class of Wasserstein DRO problems~\cite{blanchet2018optimal,sinha2017certifying}. Nevertheless, the results in~\cite{blanchet2018optimal,sinha2017certifying} do not apply to the $\ell_p$-DRSVM problem~\eqref{eq:drsvm_our}, as they require $\kappa=\infty$; i.e., the labels are error-free. Moreover, the transport cost $d_p$ does not satisfy the strong convexity-type condition in~\cite[Assumption 1]{blanchet2018optimal} or \cite[Assumption A]{sinha2017certifying}. On another front, the authors of~\cite{li2019first} introduced an ADMM-based first-order algorithmic framework to deal with the Wasserstein distributionally robust logistic regression problem. Though the framework in~\cite{li2019first} can be extended to handle the $\ell_p$-DRSVM problem~\eqref{eq:drsvm_our}, it has two main drawbacks. First, under the framework, the optimal $\lambda^*$ of problem~\eqref{eq:drsvm_our} is found by an one-dimensional search, where each update involves fixing $\lambda$ to a given value and solving for the corresponding optimal $w^*(\lambda)$ (which we refer to as the \emph{$w$-subproblem}). Since the number of $w$-subproblems that arise during the search can be large, the framework is computationally rather demanding. Second, the $w$-subproblem is solved by an ADMM-type algorithm, which involves both primal and dual updates. In order to establish fast (e.g., linear) convergence rate guarantee for the algorithm, one typically requires a regularity condition on the set of primal-dual optimal pairs of the problem at hand. Unfortunately, it is not clear whether the $\ell_p$-DRSVM problem~\eqref{eq:drsvm_our} satisfies such a primal-dual regularity condition.
	
	% Thus, multiple matrix-vector product operations have to pay for every iteration from the computational aspect.
	
	%(5) our main contribution to overcome the above issues
	To overcome these drawbacks, we propose two new epigraphical projection-based incremental algorithms for solving the $\ell_p$-DRSVM problem~\eqref{eq:drsvm_our}, which tackle the variables $(w,\lambda)$ jointly. We focus on the commonly used $\ell_1$, $\ell_2$, and $\ell_\infty$ norm-induced transport costs, which correspond to $q\in\{1,2,\infty\}$. Our first algorithm is the incremental projected subgradient descent (ISG) method, whose efficiency inherits from that of the projection onto the epigraph $\{(w,\lambda):\|w\|_q \le \lambda\}$ of the $\ell_q$ norm (with $q\in\{1,2,\infty\}$). The second is the incremental proximal point algorithm (IPPA). Although in general IPPA is less sensitive to the choice of initial step size and can achieve better accuracy than ISG~\cite{li2019incremental}, in the context of the $\ell_p$-DRSVM problem~\eqref{eq:drsvm_our}, each iteration of IPPA requires solving the following subproblem, which we refer to as the \emph{single-sample proximal point update}:
	\begin{equation}
	\min\limits_{w,\lambda} \max \left\{1-w^Tz_i, 1+w^Tz_i-\lambda\kappa,0\right\} + \frac{1}{2\alpha}(\|w-\bar{w}\|_2^2 + (\lambda-\bar{\lambda})^2),~\ \text{s.t.} \ \|w\|_q  \leq \lambda.
	%\begin{aligned}
	%& \min\limits_{w,\lambda}  \lambda\epsilon + \max \left\{1-w^Tz_i, 1+w^Tz_i-\lambda\kappa,0\right\} + \tfrac{1}{2\alpha}(\|w-\bar{w}\|_2^2 + (\lambda-\bar{\lambda})^2),~\ \text{s.t.} \ \|w\|_q  \leq \lambda.
	%\end{aligned}
	\label{eq:drsvm_ppa}
	\end{equation}
	Here, $\alpha>0$ is the step size, $q \in\{1,2,\infty\}$, and $\bar{w},\bar{\lambda}$ are given. By carefully exploiting the problem structure, we develop exceptionally efficient solutions to~\eqref{eq:drsvm_ppa}. Specifically, we show in Section~\ref{sec:ippa-sub} that the optimal solution to~\eqref{eq:drsvm_ppa} admits an analytic form when $q=2$ and can be computed by a fast algorithm based on a parametric approach and a modified secant method (cf.~\cite{dai2006new}) when $q=1$ or $\infty$.
	
	Next, we investigate the convergence behavior of the proposed ISG and IPPA when applied to problem~\eqref{eq:drsvm_our}. Our main tool is the following regularity notion:
	
	\begin{definition}[Hölderian growth condition~\cite{bolte2017error}] A function $f:\mathbb{R}^m \rightarrow \mathbb{R}$ is said to satisfy a \emph{Hölderian growth condition} on the domain $\Omega\subseteq\mathbb{R}^m$ if there exist constants $\theta \in [0,1]$ and $\sigma >0$ such that
		\begin{equation}
		\dist(x,\mathcal{X}) \leq \sigma^{-1} (f(x)-f^*)^\theta, \quad \forall x\in \Omega,
		\label{eq:HGC}
		\end{equation}
		where $\mathcal{X}$ denotes the optimal set of $\min_{x\in \Omega}f(x)$ and $f^*$ is the optimal value. The condition~\eqref{eq:HGC} is known as \emph{sharpness} when $\theta=1$ and \emph{quadratic growth} (QG) when $\theta = \frac{1}{2}$; see, e.g.,~\cite{BS00}.
		\label{def:HEB}
	\end{definition}

	We show that for different choices of $q\in\{1,2,\infty\}$ and $c\ge0$, the DRSVM problem~\eqref{eq:drsvm_our} satisfies either the sharpness condition or QG condition; see Table~\ref{tbl:compare}. With the exception of the case $q\in\{1,\infty\}$, where the sharpness (resp.~QG) of~\eqref{eq:drsvm_our} when $c=0$ (resp.~$c>0$) essentially follows from~\cite[Theorem 3.5]{burke1993weak} (resp.~\cite[Proposition 6]{zhou2017unified}), the results on the Hölderian growth of problem~\eqref{eq:drsvm_our} are new. Consequently, by choosing step sizes that decay at a suitable rate, we establish, for the first time, the fast sublinear (i.e., $\mathcal{O}(\tfrac{1}{k})$) or linear (i.e., $\mathcal{O}(\rho^k)$) convergence rate of the proposed incremental algorithms when applied to the DRSVM problem~\eqref{eq:drsvm_our}; see Table~\ref{tbl:compare}.

	\begin{table}[]
		\centering
		\caption{Convergence rates of incremental algorithms  for $\ell_p$-DRSVM}
		\label{tbl:compare}
		\begin{tabular}{lllll}
			\toprule
			$q$ & $c$ & Hölderian growth & Step size scheme & Convergence rate \\
			\midrule
			$q =1 ,\infty$&$c=0$ &\textbf{Sharp} \cite[Theorem 3.5]{burke1993weak} & $\alpha_{k+1} = \rho \alpha_k$, $\rho \in (0,1)$ & $\mathcal{O}(\rho^k)$  \\
			$q =1 ,\infty$&$c>0$ &\textbf{QG} \cite[Proposition 6]{zhou2017unified} &$\alpha_{k} =\frac{\gamma}{nk}$, $\gamma>0$&$\mathcal{O}(\frac{1}{k})$ \\
			\multirow{2}{*}{$q=2$} &\multirow{2}{*}{$c=0$} & \textbf{Sharp {(\color{blue}{BLR}})} & $\alpha_{k+1} = \rho \alpha_k,\rho \in (0,1)$ & $\mathcal{O}(\rho^k)$\\
			& & \textbf{Not Known} &$\alpha_{k} =\frac{\gamma}{n\sqrt{k}}$, $\gamma>0$ & $\mathcal{O}(\frac{1}{\sqrt{k}})$\\	
			\multirow{2}{*}{$q=2$} &\multirow{2}{*}{$c>0$} & \textbf{QG {(\color{blue}{BLR}})} & $\alpha_{k} =\frac{\gamma}{nk}$, $\gamma>0$&$\mathcal{O}(\frac{1}{k})$\\
			& & \textbf{Not Known} &$\alpha_{k} =\frac{\gamma}{n\sqrt{k}}$, $\gamma>0$ & $\mathcal{O}(\frac{1}{\sqrt{k}})$\\	
			\bottomrule
		\end{tabular}
		\begin{tablenotes}
		\item[1] BLR: The result holds under the assumption of bounded linear regularity (BLR) (see Definition~\ref{def:blr}).
		\item[2] Not Known: Without BLR, it is not known whether the Hölderian growth condition holds.
		\end{tablenotes}
	\end{table}
	
	Lastly, we demonstrate the efficiency of our proposed methods through extensive numerical experiments on both synthetic and real data sets. It is worth mentioning that our proposed algorithms can be easily extended to an asynchronous decentralized parallel setting and thus can further meet the requirements of large-scale applications.
	
	%Most of the efforts have been devoted to developing exact tractable reformulations and then tackled by the off-the-shelf solvers additionally ~\cite{shafieezadeh2015distributionally,shafieezadeh2019regularization,esfahani2018data,blanchet2018optimal,wiesemann2014distributionally,kuhn2019wasserstein}.
	
	%There are a few works in the literature that develops fast iterative algorithms to address the class of Wasserstein DRO problems.
	%Lately, \cite{sinha2017certifying} has proposed stochastic (sub)gradient descent algorithms for DRO problem with Wasserstein metric.
	%The $\ell_p$ norm-based transport cost considered in this paper cannot be covered nonetheless (i.e., see Assumption A in~\cite{sinha2017certifying}).
	%In a similar work, \cite{blanchet2018optimal} is also restricted by the quadratic transport cost form and its strongly convex assumption. Moreover, both of them can just handle the noiseless case (i.e., $\kappa =
	%\infty$).
	%By far, the most relevant paper to this work is the concurrent work of~\cite{li2019first} and operator-splitting scheme based iterative algorithms have been proposed to tackle the logistic regression case.
	%However, we develop an entirely different primal-only algorithmic framework to address another particular loss function class (e.g., hinge loss), which is complementary with~\cite{li2019first}  and potentially more effective for real applications.
%	\vspace{-0.5\baselineskip}
	\section{Epigraphical Projection-based Incremental Algorithms}
%	\vspace{-0.5\baselineskip}
	In this section, we present our incremental algorithms for solving the $\ell_p$-DRSVM problem. For simplicity, we focus on the case $c =0$ in what follows. Our technical development can be extended to handle the general case $c\ge0$ by noting that the subproblems corresponding to the cases $c=0$ and $c>0$ share the same structure.
	
	To begin, observe that the $\ell_p$-DRSVM problem~\eqref{eq:drsvm_our} with $c=0$ can be written compactly as
	\begin{equation}
	\label{eq:dro_ge}
	\min_{\|w\|_q\leq \lambda} \frac{1}{n}\sum_{i=1}^n f_i(w,\lambda),
	\end{equation}
	where $f_i(w,\lambda)=\lambda\epsilon + \max \left\{1-w^Tz_i, 1+w^Tz_i-\lambda\kappa,0\right\}$ is a piecewise affine function. Since problem \eqref{eq:dro_ge} possesses the vanilla finite-sum structure with a single epigraphical projection constraint, a natural and widely adopted approach to tackling it is to use incremental algorithms. Roughly speaking, such algorithms select one mini-batch of component functions from the objective in~\eqref{eq:dro_ge} at a time based on a certain cyclic order and use the selected functions to update the current iterate. We shall focus on the following two incremental algorithms for solving the DRSVM problem~\eqref{eq:drsvm_our}. Here, $k$ is the epoch index (i.e., the $k$-th time going through the cyclic order) and $\alpha_k>0$ is the step size in the $k$-th epoch.
	
	%Namely, in each epoch, incremental algorithms update the iterate with only one component function (i.e., single sample / mini-batch based) selected based on a cyclic order.
	\paragraph{Incremental Mini-batch Projected Subgradient Algorithm (ISG)}
	\begin{equation}
	\label{eq:ISG}
	(w_{i+1}^{k},\lambda_{i+1}^{k}) = \proj_{\{\|w\|_q \leq \lambda\}}\left[(w_i^{k},\lambda_i^{k}) - \alpha_k g_i^k \right],
	\end{equation}
	where $g_i^k$ is a subgradient of $\frac{1}{|B_i|}\sum_{j\in B_i}f_j$ at $(w_i^{k},\lambda_i^{k})$ and $B_i \subseteq\{1,\ldots,n\}$ is the $i$-th mini-batch.
	\paragraph{Incremental Proximal Point Algorithm (IPPA)}
	\begin{equation}
	\label{eq:IPPA}
	(w_{i+1}^{k},\lambda_{i+1}^{k}) = \mathop{\arg\min}_{\|w\|_q \leq \lambda} \left\{
	f_i(w,\lambda) + \frac{1}{2\alpha_k}\left(\|w-w^k_i\|_2^2 +(\lambda-\lambda^k_i)^2\right)
	\right\}, %\footnote{Note that $w_{n}^k = w_0^{k+1};$}.
	\end{equation}
	where $(w_{n}^k,\lambda_n^k) = (w_0^{k+1},\lambda_0^{k+1})$.
	
	Now, a natural question is how to solve the subproblems \eqref{eq:ISG} and \eqref{eq:IPPA} efficiently. As it turns out, the key lies in an efficient implementation of the $\ell_q$ norm epigraphical projection (with $q\in\{1,2,\infty\}$). Indeed, such a projection appears explicitly in the ISG update~\eqref{eq:ISG} and, as we shall see later, plays a vital role in the design of fast iterative algorithms for the single-sample proximal point update \eqref{eq:IPPA}. To begin, we note that the $\ell_2$ norm epigraphical projection $\proj_{\{\|w\|_2\le\lambda\}}$ has a well-known analytic solution; see~\cite[Theorem 3.3.6]{bauschke1996projection}. Next, the $\ell_1$ norm epigraphical projection $\proj_{\{\|w\|_1\le\lambda\}}$ can be found in linear time using the quick-select algorithm; see~\cite{wang2016epigraph}. Lastly, the $\ell_{\infty}$ norm epigraphical projection $\proj_{\{\|w\|_{\infty}\le\lambda\}}$ can be computed in linear time via the Moreau decomposition
	\[ \proj_{\{\|w\|_{\infty}\le\lambda\}}(x,s) = (x,s)+\proj_{\{\|w\|_1 \le \lambda\}}(-x,-s). \]
	From the above discussion, we see that the ISG update~\eqref{eq:ISG} can be computed efficiently. In the next section, we discuss how these epigraphical projections can be used to perform the single-sample proximal point update \eqref{eq:IPPA} in an efficient manner.
	
	%\vspace{-0.5\baselineskip}
	\section{Fast Algorithms for Single-Sample Proximal Point Update~\eqref{eq:IPPA}} \label{sec:ippa-sub}
	%\vspace{-0.5\baselineskip}
	%In this section, we present the details of solving a single sample proximal point update \eqref{eq:drsvm_ppa} efficiently.
	\paragraph{Analytic solution for $q=2$.} We begin with the case $q=2$. By combining the terms $\lambda\epsilon$ and $\tfrac{1}{2\alpha_k}(\lambda-\lambda_i^k)^2$ in~\eqref{eq:IPPA}, we see that the single-sample proximal point update takes the form (cf.~\eqref{eq:drsvm_ppa})
	\begin{equation}
	\min\limits_{w,\lambda}  \underbrace{\max \left\{1-w^Tz_i, 1+w^Tz_i-\lambda\kappa,0\right\}}_{h_i(w,\lambda)} + \frac{1}{2\alpha}(\|w-\bar{w}\|_2^2 + (\lambda-\bar{\lambda})^2),~\ \text{s.t.} \ \|w\|_2 \leq \lambda.
	% \begin{aligned}
	% & \min\limits_{w,\lambda}  \max \left\{1-w^Tz_i, 1+w^Tz_i-\lambda\kappa,0\right\} + \tfrac{1}{2\alpha}(\|w-\bar{w}\|_2^2 + (\lambda-\bar{\lambda})^2),~\ \text{s.t.} \ \|w\|_2  \leq \lambda.
	% \end{aligned}
	\label{eq:l2}
	\end{equation}
	The main difficulty of~\eqref{eq:l2} lies in the piecewise affine term $h_i$. To handle this term, let $h_{i,1}(w,\lambda) = 1-w^Tz_i$, $h_{i,2}(w,\lambda) = 1+w^Tz_i-\lambda\kappa$, and $h_{i,3}(w,\lambda)=0$, so that $h_i = \max_{j\in\{1,2,3\}} h_{i,j}$. Observe that if $(w^*,\lambda^*)$ is an optimal solution to~\eqref{eq:l2}, then there could only be one, two, or three affine pieces in $h_i$ that are active at $(w^*,\lambda^*)$; i.e., $\Gamma = |\{j:h_i(w^*,\lambda^*) = h_{i,j}(w^*,\lambda^*)\}| \in \{1,2,3\}$. This suggests that we can find $(w^*,\lambda^*)$ by exhausting these possibilities. Due to space limitation, we only give an outline of our strategy here. The details can be found in the Appendix.
	
	We start with the case $\Gamma=1$. For $j=1,2,3$, consider the following problem, which corresponds to the subcase where $h_{i,j}$ is the only active affine piece:
	\begin{equation} \label{eq:case1}
	\min\limits_{w,\lambda}  h_{i,j}(w,\lambda)+ \frac{1}{2\alpha}(\|w-\bar{w}\|_2^2 + (\lambda-\bar{\lambda})^2),~\ \text{s.t.} \ \|w\|_2 \leq \lambda.
	\end{equation}
	Since $h_{i,j}$ is affine in $(w,\lambda)$, it is easy to verify that problem~\eqref{eq:case1} reduces to an $\ell_2$ norm epigraphical projection, which admits an analytic solution, say $(\hat{w}_j,\hat{\lambda}_j)$. If there exists a $j'\in\{1,2,3\}$ such that $h_{i,j'}(\hat{w}_{j'},\hat{\lambda}_{j'}) > h_{i,j}(\hat{w}_{j'},\hat{\lambda}_{j'})$ for $j\not=j'$, then we know that $(\hat{w}_{j'},\hat{\lambda}_{j'})$ is optimal for~\eqref{eq:l2} and hence we can terminate the process. Otherwise, we proceed to the case $\Gamma=2$ and consider, for $1\le j<j'\le 3$, the following problem, which corresponds to the subcase where $h_{i,j}$ and $h_{i,j'}$ are the only two active affine pieces:
	\begin{equation} \label{eq:case2}
	\min_{w,\lambda} h_{i,j}(w,\lambda) + \frac{1}{2\alpha} ( \|w-\bar{w}\|_2^2 + (\lambda -  \bar{\lambda})^2 ),~\ \text{s.t.}~h_{i,j}(w,\lambda) = h_{i,j'}(w,\lambda),~\|w\|_2 \leq \lambda.
	\end{equation}
	As shown in the Appendix (Proposition~\ref{prop:l2sub}), the optimal solution to~\eqref{eq:case2} can be found by solving a univariate quartic equation, which can be done efficiently. Now, let $(\hat{w}_{(j,j')},\hat{\lambda}_{(j,j')})$ be the optimal solution to~\eqref{eq:case2}. If there exist $j,j'$ with $1\le j<j'\le 3$ such that $h_{i,j}(\hat{w}_{(j,j')},\hat{\lambda}_{(j,j')}) = h_{i,j'}(\hat{w}_{(j,j')},\hat{\lambda}_{(j,j')}) > h_{i,j''}(\hat{w}_{(j,j')},\hat{\lambda}_{(j,j')})$ with $j'' \in \{1,2,3\}\setminus\{j,j'\}$, then $(\hat{w}_{(j,j')},\hat{\lambda}_{(j,j')})$ is optimal for~\eqref{eq:l2} and we can terminate the process. Otherwise, we proceed to the case $\Gamma=3$. In this case, we consider the problem
	\[ \min_{w,\lambda} \frac{1}{2\alpha} ( \|w-\bar{w}\|_2^2 + (\lambda -  \bar{\lambda})^2 ),~\ \text{s.t.}~h_{i,1}(w,\lambda) = h_{i,2}(w,\lambda) = h_{i,3}(w,\lambda),~\|w\|_2 \leq \lambda, \]
	which reduces to
	\begin{equation} \label{eq:case3}
	\min_w \frac{1}{2\alpha}\|w-\bar{w}\|_2^2,~\ \text{s.t.}~w^Tz_i = 1, \, \|w\|_2 \le \frac{2}{\kappa}.
	\end{equation}
	It can be shown that problem~\eqref{eq:case3} admits an analytic solution $\hat{w}$; see the Appendix (Proposition~\ref{prop:l2sub2}). Then, the pair $(\hat{w},\tfrac{2}{\kappa})$ is an optimal solution to~\eqref{eq:l2}.

	\paragraph{Fast iterative algorithm for $q=1$.} The high-level idea is similar to that for the case $q=2$; i.e., we systematically go through all valid subcollections of the affine pieces in $h_i$ and test whether they can be active at the optimal solution to the single-sample proximal point update. The main difference here is that the subproblems arising from the subcollections do not necessarily admit analytic solutions. To overcome this difficulty, we propose a modified secant algorithm (cf.~\cite{dai2006new}) to search for the optimal dual multiplier of the subproblem and use it to recover the optimal solution to the original subproblem via $\ell_1$ norm epigraphical projection. Again, we give an outline of our strategy here and relegate the details to the Appendix.
	
	To begin, we rewrite the single-sample proximal point update~\eqref{eq:IPPA} for the case $q=1$ as 
	\begin{equation}
	\label{eq:l1ppa}
	\begin{aligned}
	& \min_{w,\lambda,\mu}  \,\, \mu + \frac{1}{2\alpha} \left(\|w-\bar{w}\|_2^2 + (\lambda- \bar{\lambda})^2\right) \\
%	& \,\,\,\,\text{s.t.}~\,\,\, h_{i,j}(w,\lambda) \le \mu \,\, {\color{blue} ( \leftarrow \sigma_j \ge 0 )}, \,\, j=1,2,3; \,\,\, \|w\|_1 \leq \lambda. % \\
	& \,\,\,\,\text{s.t.}~\,\,\, h_{i,j}(w,\lambda) \le \mu, \,\, j=1,2,3; \,\,\, \|w\|_1 \leq \lambda.
	% 1+w^Tz_i-\lambda\kappa \leq \mu, \quad{\color{blue}{\leftarrow \sigma_1 \ge 0 }} \\
	%	& \,\,\,\,{\color{white}{\text{s.t.}}}~\,\,\, 1-w^Tz_i \leq \mu, \quad{\color{blue}{\leftarrow \sigma_2 \ge 0 }} \\
	%	& \,\,\,\,{\color{white}{\text{s.t.}}}~\,\,\, 0 \leq \mu, \quad{\color{blue}{\leftarrow \sigma_3 \ge 0 }}\\
	%	& \,\,\,\,{\color{white}{\text{s.t.}}}~\,\,\, \|w\|_1 \leq \lambda,
	\end{aligned}
	\end{equation}
	%where $\sigma_1,\sigma_2,\sigma_3\ge0$ are the corresponding dual multipliers. 
	For reason that would become clear in a moment, we shall not go through the cases $\Gamma=1,2,3$ as before. Instead, consider first the case where $h_{i,3}$ is inactive. If $h_{i,1}$ is also inactive, then we consider the problem $\min_{w,\lambda} h_{i,2}(w,\lambda) + \tfrac{1}{2\alpha} \left(\|w-\bar{w}\|_2^2 + (\lambda- \bar{\lambda})^2\right)$, which, by the affine nature of $h_{i,2}$, is equivalent to an $\ell_1$ norm epigraphical projection. If $h_{i,1}$ is active, then we consider the problem
	\begin{equation} \label{eq:l1-c1}
	\min_{w,\lambda} h_{i,1}(w,\lambda) + \frac{1}{2\alpha} ( \|w-\bar{w}\|_2^2 + (\lambda -  \bar{\lambda})^2 ),~\ \text{s.t.}~h_{i,2}(w,\lambda) \le h_{i,1}(w,\lambda),~\|w\|_1 \leq \lambda.
	\end{equation}
	Note that $h_{i,2}$ can be active or inactive, and the constraint $h_{i,2}(w,\lambda) \le h_{i,1}(w,\lambda)$ allows us to treat both possibilities simultaneously. Hence, we do not need to tackle them separately as we did in the case $q=2$. Observe that problem~\eqref{eq:l1-c1} can be cast into the form 
	\begin{equation} 	\label{eq:l1ppa_sub}
	\min_{w,\lambda} \frac{1}{2\alpha} ( \|w-\bar{w}\|_2^2 + (\lambda -  \bar{\lambda})^2 ),~\ \text{s.t.}~w^Tz \le a\lambda + b \,\,{\color{blue} (\leftarrow \sigma\ge0)},\,\,\,\|w\|_1 \leq \lambda,
	\end{equation}
	where, with an abuse of notation, we use $\bar{w}\in\mathbb{R}^d$, $\bar{\lambda}\in\mathbb{R}$ here again and caution the reader that they are different from those in~\eqref{eq:l1-c1}, and $z=z_i$, $a=\tfrac{\kappa}{2}$, $b=0$. Before we discuss how to solve the subproblem~\eqref{eq:l1ppa_sub}, let us note that it arises in the case where $h_{i,3}$ is active as well. Indeed, if $h_{i,3}$ is active and $h_{i,1}$ is inactive, then we have $z=z_i$, $a=\kappa$, $b=-1$, which corresponds to the constraint $h_{i,2}(w,\lambda)\le h_{i,3}(w,\lambda)$ and covers the possibilities that $h_{i,2}$ is active and inactive. On the other hand, if $h_{i,3}$ is active and $h_{i,2}$ is inactive, then we have $z=-z_i$, $a=0$, $b=-1$, which corresponds to the constraint $h_{i,1}(w,\lambda) \le h_{i,3}(w,\lambda)$ and covers the possibilities that $h_{i,1}$ is active and inactive. The only remaining case is when $h_{i,1},h_{i,2},h_{i,3}$ are all active. In this case, we consider problem~\eqref{eq:case3} with $\|w\|_2\le\tfrac{2}{\kappa}$ replaced by $\|w\|_1\le\tfrac{2}{\kappa}$. As shown in the Appendix, such a problem can be tackled using the technique for solving~\eqref{eq:l1ppa_sub}. We go through the above cases sequentially and terminate the process if the solution to the subproblem in any one of the cases satisfies the optimality conditions of~\eqref{eq:l1ppa}. 

	Now, let us return to the issue of solving~\eqref{eq:l1ppa_sub}. The main idea is to perform an one-dimensional search on the dual variable $\sigma$ to find the optimal dual multiplier $\sigma^*$. Specifically, consider the problem % based on the efficient implementation of the $\ell_1$ epigraphical projection,
	\begin{eqnarray}
	\min_{\|w\|_1 \leq \lambda} & \displaystyle \frac{1}{2\alpha} \left(\|w-\bar{w}\|_2^2 + (\lambda- \bar{\lambda})^2\right)  + \sigma (w^Tz - a\kappa - b).
	\label{eq:secant}
	\end{eqnarray} 		
	Let $(\hat{w}(\sigma),\hat{\lambda}(\sigma))$ be the optimal solution to \eqref{eq:secant} and define the function $p:\mathbb{R}_+\rightarrow\mathbb{R}$ by $p(\sigma) =  \hat{w}(\sigma)^Tz - a\kappa - b$. Inspired by~\cite{liu2017fast}, we establish the following monotonicity property of $p$, which will be crucial to our development of an extremely efficient algorithm for solving~\eqref{eq:l1ppa_sub} later. 
	\begin{proposition}
		If $\sigma$ satisfies (i) $\sigma = 0$ and $p(\sigma)\leq 0$, or (ii) $p(\sigma) =0$, then $(\hat{w}(\sigma),\hat{\lambda}(\sigma))$ is the optimal solution to \eqref{eq:l1ppa_sub}. Moreover, $p$ is continuous and monotonically non-increasing on $\mathbb{R}_+$.
		\label{prop:sigma}
	\end{proposition}
	In view of Proposition~\ref{prop:sigma}, we first check if $p(0)\le0$ via an $\ell_1$ norm epigraphical projection. If not, then we search for the $\sigma^*\ge0$ that satisfies $p(\sigma^*)=0$ by the secant method, with some special modifications designed to speed up its convergence~\cite{dai2006new}. Let us now give a high-level description of our modified secant method. We refer the reader to the Appendix (Algorithm~\ref{algo:MSA}) for details.

	At the beginning of a generic iteration of the method, we have an interval $[\sigma_l,\sigma_u]$ that contains $\sigma^*$, with $r_l = -p(\sigma_l)<0$ and $r_u = -p(\sigma_u)>0$. The initial interval can be found by considering the optimality conditions of~\eqref{eq:l1ppa} (i.e., $\sigma^* \in [0,1]$). We then take a secant step to get a new point $\sigma$ with $r=-p(\sigma)$ and perform the update on $\sigma_l,\sigma_u$ as follows. 
	
	%	More importantly, the bracket phase in ~\cite{dai2006new,liu2017fast} can be omitted to avoid the extra computational cost. Instead, we can derive its range by the Karush-Kuhn-Tucker (KKT) conditions of \eqref{eq:l1ppa}. Namely, if this sub-case is optimal for \eqref{eq:l1ppa}, we can match these two KKT systems (i.e., LHS of \eqref{eq:l1ppa_sub} and \eqref{eq:l1ppa})  and finally obtain $\sigma_1^* \in [0,1]$ (i.e., $\sigma_1^* + \sigma_3^* = 1$, $\sigma_1^* \ge 0$ and $\sigma_3^*\ge 0$). 

\begin{minipage}[b]{0.67\linewidth}	 
 Suppose that $r>0$. If $\sigma$ lies in the left-half of the interval (i.e., $\sigma<\frac{\sigma_l+\sigma_u}{2}$), then we update $\sigma_u$ to $\sigma$. Otherwise, we take an auxiliary secant step based on $\sigma$ and $\sigma_u$ to get a point $\sigma'$, and we update $\sigma_u$ to $\max\{ \sigma', 0.6\sigma_l+0.4\sigma \}$. Such a choice ensures that the interval length is reduced by a factor of $0.6$ or less. The case where $r<0$ is similar, except that $\sigma_l$ is updated. If $r=0$, then by Proposition~\ref{prop:sigma} we have found the optimal dual multiplier $\sigma^*$ and hence can terminate.
	\end{minipage}
	\begin{minipage}[b]{0.32\linewidth}
		\tikzset{every picture/.style={line width=0.4pt}} %set default line width to 0.75pt
		\begin{tikzpicture}[x=0.45pt,y=0.45pt,yscale=-0.87,xscale=1]
		
		%Shape: Axis 2D [id:dp6980124238063423]
		\draw  (84.32,187.84) -- (352.32,187.84)(120.5,78.6) -- (120.5,262.6) (345.32,182.84) -- (352.32,187.84) -- (345.32,192.84) (115.5,85.6) -- (120.5,78.6) -- (125.5,85.6)  ;
		%Shape: Boxed Bezier Curve [id:dp3966395011659791]
		\draw    (120.5,241.46) .. controls (132.75,116.41) and (310.35,149.83) .. (317.5,111.51) ;
		%Straight Lines [id:da3876233939117333]
		\draw  [dash pattern={on 0.84pt off 2.51pt}]  (317.5,111.51) -- (318.5,185.51) ;
		%Shape: Circle [id:dp8177870325364229]
		\draw  [color={rgb, 255:red, 208; green, 2; blue, 27 }  ,draw opacity=1 ][fill={rgb, 255:red, 208; green, 2; blue, 27 }  ,fill opacity=1 ] (126,214.51) .. controls (126,213.4) and (126.9,212.51) .. (128,212.51) .. controls (129.1,212.51) and (130,213.4) .. (130,214.51) .. controls (130,215.61) and (129.1,216.51) .. (128,216.51) .. controls (126.9,216.51) and (126,215.61) .. (126,214.51) -- cycle ;
		%Shape: Circle [id:dp7509874636923195]
		\draw  [color={rgb, 255:red, 208; green, 2; blue, 27 }  ,draw opacity=1 ][fill={rgb, 255:red, 208; green, 2; blue, 27 }  ,fill opacity=1 ] (266,134.51) .. controls (266,133.4) and (266.9,132.51) .. (268,132.51) .. controls (269.1,132.51) and (270,133.4) .. (270,134.51) .. controls (270,135.61) and (269.1,136.51) .. (268,136.51) .. controls (266.9,136.51) and (266,135.61) .. (266,134.51) -- cycle ;
		%Straight Lines [id:da2585248142503169]
		\draw [color={rgb, 255:red, 208; green, 2; blue, 27 }  ,draw opacity=1 ] [dash pattern={on 4.5pt off 4.5pt}]  (130,214.51) -- (268,134.51) ;
		%Shape: Circle [id:dp31538603586428415]
		\draw  [color={rgb, 255:red, 208; green, 2; blue, 27 }  ,draw opacity=1 ][fill={rgb, 255:red, 208; green, 2; blue, 27 }  ,fill opacity=1 ] (177,187.51) .. controls (177,186.4) and (177.9,185.51) .. (179,185.51) .. controls (180.1,185.51) and (181,186.4) .. (181,187.51) .. controls (181,188.61) and (180.1,189.51) .. (179,189.51) .. controls (177.9,189.51) and (177,188.61) .. (177,187.51) -- cycle ;
		%Straight Lines [id:da04642297022272435]
		\draw [color={rgb, 255:red, 208; green, 2; blue, 27 }  ,draw opacity=1 ] [dash pattern={on 0.84pt off 2.51pt}]  (268,136.51) -- (268.5,187.84) ;
		%Straight Lines [id:da32162774029959107]
		\draw [color={rgb, 255:red, 208; green, 2; blue, 27 }  ,draw opacity=1 ] [dash pattern={on 0.84pt off 2.51pt}]  (128.5,187.84) -- (128,212.51) ;
		%Shape: Circle [id:dp9009663756328732]
		\draw  [color={rgb, 255:red, 0; green, 0; blue, 0 }  ,draw opacity=1 ][fill={rgb, 255:red, 0; green, 0; blue, 0 }  ,fill opacity=1 ] (197,187.51) .. controls (197,186.4) and (197.9,185.51) .. (199,185.51) .. controls (200.1,185.51) and (201,186.4) .. (201,187.51) .. controls (201,188.61) and (200.1,189.51) .. (199,189.51) .. controls (197.9,189.51) and (197,188.61) .. (197,187.51) -- cycle ;
		%Shape: Circle [id:dp7778391185341162]
		\draw  [color={rgb, 255:red, 208; green, 2; blue, 27 }  ,draw opacity=1 ][fill={rgb, 255:red, 208; green, 2; blue, 27 }  ,fill opacity=1 ] (126.5,187.84) .. controls (126.5,186.73) and (127.4,185.84) .. (128.5,185.84) .. controls (129.6,185.84) and (130.5,186.73) .. (130.5,187.84) .. controls (130.5,188.94) and (129.6,189.84) .. (128.5,189.84) .. controls (127.4,189.84) and (126.5,188.94) .. (126.5,187.84) -- cycle ;
		%Shape: Circle [id:dp5125261110682982]
		\draw  [color={rgb, 255:red, 208; green, 2; blue, 27 }  ,draw opacity=1 ][fill={rgb, 255:red, 208; green, 2; blue, 27 }  ,fill opacity=1 ] (266.5,187.84) .. controls (266.5,186.73) and (267.4,185.84) .. (268.5,185.84) .. controls (269.6,185.84) and (270.5,186.73) .. (270.5,187.84) .. controls (270.5,188.94) and (269.6,189.84) .. (268.5,189.84) .. controls (267.4,189.84) and (266.5,188.94) .. (266.5,187.84) -- cycle ;
		%Straight Lines [id:da9246990274071922]
		\draw [color={rgb, 255:red, 208; green, 2; blue, 27 }  ,draw opacity=1 ] [dash pattern={on 0.84pt off 2.51pt}]  (178.5,158.46) -- (179,185.51) ;
		%Shape: Circle [id:dp9188050649736601]
		\draw  [color={rgb, 255:red, 208; green, 2; blue, 27 }  ,draw opacity=1 ][fill={rgb, 255:red, 208; green, 2; blue, 27 }  ,fill opacity=1 ] (176.5,160.46) .. controls (176.5,159.35) and (177.4,158.46) .. (178.5,158.46) .. controls (179.6,158.46) and (180.5,159.35) .. (180.5,160.46) .. controls (180.5,161.56) and (179.6,162.46) .. (178.5,162.46) .. controls (177.4,162.46) and (176.5,161.56) .. (176.5,160.46) -- cycle ;
		
		% Text Node
		\draw (185,191.51) node [anchor=north west][inner sep=0.75pt]  [font=\scriptsize,color={rgb, 255:red, 0; green, 0; blue, 0 }  ,opacity=1 ]  {$\frac{\sigma _{l} +\sigma _{u}}{2}$};
		% Text Node
		\draw (162,138) node [anchor=north west][inner sep=0.75pt]  [font=\scriptsize,color={rgb, 255:red, 208; green, 2; blue, 27 }  ,opacity=1 ]  {$r >0$};
		% Text Node
		\draw (312,189) node [anchor=north west][inner sep=0.75pt]  [font=\scriptsize,color={rgb, 255:red, 0; green, 0; blue, 0 }  ,opacity=1 ]  {$1$};
		% Text Node
		\draw (105,195) node [anchor=north west][inner sep=0.75pt]  [font=\scriptsize,color={rgb, 255:red, 0; green, 0; blue, 0 }  ,opacity=1 ]  {$0$};
		% Text Node
		\draw (258,191.51) node [anchor=north west][inner sep=0.75pt]  [font=\scriptsize,color={rgb, 255:red, 208; green, 2; blue, 27 }  ,opacity=1 ]  {$\sigma _{u} \ $};
		% Text Node
		\draw (120.5,191.84) node [anchor=north west][inner sep=0.75pt]  [font=\scriptsize,color={rgb, 255:red, 208; green, 2; blue, 27 }  ,opacity=1 ]  {$\sigma _{l} \ $};
		% Text Node
		\draw (173,191.51) node [anchor=north west][inner sep=0.75pt]  [font=\scriptsize,color={rgb, 255:red, 208; green, 2; blue, 27 }  ,opacity=1 ]  {$\sigma $};
		% Text Node
		\draw (245,108) node [anchor=north west][inner sep=0.75pt]  [font=\scriptsize,color={rgb, 255:red, 208; green, 2; blue, 27 }  ,opacity=1 ]  {$( \sigma _{u} ,r_{u})$};
		% Text Node
		\draw (126,214.51) node [anchor=north west][inner sep=0.75pt]  [font=\scriptsize,color={rgb, 255:red, 208; green, 2; blue, 27 }  ,opacity=1 ]  {$( \sigma _{l} ,r_{l})$};
		% Text Node
		\draw (102,58) node [anchor=north west][inner sep=0.75pt]  [font=\scriptsize,color={rgb, 255:red, 0; green, 0; blue, 0 }  ,opacity=1 ]  {$-p( \sigma )$};
		% Text Node
		\draw (340,195) node [anchor=north west][inner sep=0.75pt]  [font=\scriptsize,color={rgb, 255:red, 0; green, 0; blue, 0 }  ,opacity=1 ]  {$\sigma $};
	     %\draw (215,245) node [anchor=north west][inner sep=0.75pt]  [font=\scriptsize,color={rgb, 255:red, 0; green, 0; blue, 0 }  ,opacity=1 ]  {$(\spadesuit)$};
		\end{tikzpicture}
	\end{minipage}

	Finally, for the case $q=\infty$, we can follow the same procedure as the case $q=1$. The details can be found in the Appendix. %The pseudo-code of modified secant algorithm and all sub-cases to solve \eqref{eq:l1ppa} are given in Appendix.

%\vspace{-0.5\baselineskip}
\section{Convergence Rate Analysis of Incremental Algorithms} \label{sec:analysis}
%\vspace{-0.5\baselineskip}
In this section, we study the convergence behavior of our proposed incremental methods ISG and IPPA. Our starting point is to understand the conditons under which the $\ell_p$-DRSVM problem~\eqref{eq:drsvm_our} possesses the Hölderian growth condition~\eqref{eq:HGC}. Then, by determining the value of the growth exponent $\theta$ and using it to choose step sizes that decay at a suitable rate, we can establish the convergence rates of ISG and IPPA. To begin, let us consider problem~\eqref{eq:drsvm_our} with $q\in\{1,\infty\}$. If $c=0$, then problem~\eqref{eq:drsvm_our} satisfies the sharpness (i.e., $\theta=1$) condition. This follows essentially from~\cite[Theorem 3.5]{burke1993weak}, as the objective of~\eqref{eq:drsvm_our} has polyhedral epigraph and the constraint is polyhedral. On the other hand, if $c>0$, then since the piecewise affine term in~\eqref{eq:drsvm_our} has a polyhedral epigraph and the constraint is polyhedral, we can invoke~\cite[Proposition 6]{zhou2017unified} and conclude that problem~\eqref{eq:drsvm_our} satisfies the QG (i.e., $\theta=\tfrac{1}{2}$) condition.

Next, let us consider the case $q=2$. From the above discussion, one may expect that similar conclusions hold for this case. However, as the following example shows, this case is more subtle and requires a more careful treatment.
\begin{example}  
Consider the problem $\min_{w,\lambda} 0.1\lambda + |1-w_1|,~{\rm s.t.}~\ \sqrt{w_1^2+w_2^2} \leq \lambda$, which is an instance of~\eqref{eq:drsvm_our} with $q=2$, $c=0$. It is easy to verify that the optimal solution is $w^*=(1,0)$, $\lambda^*=1$. Consider feasible points of the form $(w_1,w_2,\lambda)=(w_1,\sqrt{1-w_1^2},1)$, which tend to $(w^*,\lambda^*)$ as $w_1$ tends to $1$. A simple calculation yields $\dist((w_1,w_2,\lambda),(w^*,\lambda^*)) = \sqrt{2|1-w_1|} = \omega(|1-w_1|)$, which shows that the instance cannot satisfy the sharpness condition.
\end{example}
As it turns out, it is still possible to establish the sharpness or QG condition for problem~\eqref{eq:drsvm_our} with $q=2$ under a well-known sufficient condition called \emph{bounded linear regularity}. Let us begin with the definition.

\begin{definition}[Bounded linear regularity~{\cite[Definition 5.6]{bauschke1996projectionsiam}}]
	Let $C_1,\ldots,C_N$ be closed convex subsets of $\mathbb{R}^d$ with a non-empty intersection $C$. We say that the collection $\{C_1,\ldots,C_N\}$ is \emph{bounded linearly regular (BLR)} if for every bounded subset $\mathcal{B}$ of $\mathbb{R}^d$, there exists a constant $\kappa>0$ such that 
	\[ \dist(x,C) \leq \kappa \max_{i\in\{1,\ldots, N\}} \dist(x,C_i), ~\text{for all} ~x\in \mathcal{B}.\]
	\label{def:blr}
\end{definition}
\vspace{-6mm}
Using the above definition, we can establish the following result; see the Appendix for the proof.
\begin{proposition} Consider problem~\eqref{eq:drsvm_our} with $q=2$. Let $\mathcal{X}$ be the set of optimal solutions and $L_2^d = \{(w,\lambda)\in\mathbb{R}^d\times\mathbb{R}:\|w\|_2 \le \lambda\}$ be the constraint set. Suppose that $\mathcal{X}\cap\ri(L_2^d) \neq \emptyset$.
	%\Jiajin{ Then, the collection $\{\mathcal{X},L_2^d\}$ is BLR.}
 Consequently, problem~\eqref{eq:drsvm_our} satisfies the sharpness condition when $c=0$ and the QG condition when $c>0$.
	\label{prop:blr1}
\end{proposition}
\vspace{-3mm}
%\begin{lemma}{\cite[Corollary 3]{bauschke1999strong}}
%	Let $C_1, \cdots, C_N $ be closed convex subsets of $\mathbb{R}^d$ , where $C_{r+1}, \cdots, C_{N}$ are polyhedral for some $r \in \{0,1,\cdots,N\}$. Suppose that
%	$ \bigcap_{i=1}^r \ri(C_i)\cap\bigcap_{i=r+1}^NC_i \neq  \emptyset.$
%	Then, the collection $\{C_1,\cdots,C_N\}$ is boundedly linearly regular (BLR). 
%	\label{le:blr}
%\end{lemma}
By combining Proposition~\ref{prop:blr1} with an appropriate choice of step sizes, we obtain the following convergence rate guarantees for ISG and IPPA. The proof can be found in the Appendix.
\begin{theorem} Let $\{x^k=(w_0^k,\lambda_0^k)\}$ be the sequence of iterates generated by ISG or IPPA.
	\begin{enumerate}[(1)] 
		\item If problem \eqref{eq:drsvm_our} satisfies the \emph{sharpness} condition, then by choosing the geometrically diminishing step sizes $\alpha_k = \alpha_0 \rho^k$ with $\alpha_0 \ge \tfrac{\sigma\dist(x^{0},\mathcal{X})}{2L^2n}$ and $\sqrt{1-\tfrac{\sigma^2}{2L^2}} \leq \rho <1$, the sequence $\{x^k\}$ converges linearly to an optimal solution to~\eqref{eq:drsvm_our}; i.e., $\dist(x^k,\mathcal{X})\leq \mathcal{O}(\rho^k)$ for all $k\ge0$.
		\vspace{-2mm}
		\item If problem \eqref{eq:drsvm_our} satisfies the \emph{quadratic growth} condition, then by choosing the polynomially decaying step sizes $\alpha_k = \tfrac{\gamma}{nk}$ with $\gamma>\frac{1}{2\sigma}$, the sequence $\{x^k\}$ converges to an optimal solution to~\eqref{eq:drsvm_our} at the rate $\mathcal{O}(\tfrac{1}{\sqrt{k}})$ and $\{f(x^k)-f^*\}$ converges to zero at the rate $\mathcal{O}(\tfrac{1}{k})$. 
		\vspace{-2mm}
		\item (See~\cite[Proposition 2.10]{nedic2001convergence}) For the general convex problem~\eqref{eq:drsvm_our}, by choosing the step sizes $\alpha_k = \tfrac{\gamma}{n\sqrt{k}}$ with $\gamma>0$, the sequence $\displaystyle\{\min_{0\leq k\leq K} f(x^k) -f^*\}$ converges to zero at the rate $\mathcal{O}(\tfrac{1}{\sqrt{K}})$.
	\end{enumerate}
	\label{thm:conv}
\end{theorem}
%\vspace{-\baselineskip}
\section{Experiment Results}\label{sec:exp}
%\vspace{-0.5\baselineskip}
In this section, we present numerical results to demonstrate the efficiency of our proposed incremental methods. All simulations are implemented using MATLAB R2019b on a computer running Windows 10 with a 3.20 GHz, the Intel(R) Core(TM) i7-8700 processor, and 16 GB of RAM. 
To begin, we evaluate our two proposed incremental methods ISG and IPPA in different settings to corroborate our theoretical results in Section~\ref{sec:analysis} and to better understand their empirical strengths and weaknesses. Based on this, we develop a hybrid algorithm that combines the advantages of both ISG and IPPA to further speed up the convergence in practice. 
Next, we compare the wall-clock time of our algorithms with GS-ADMM~\cite{li2019first} and  YALMIP (i.e., \texttt{IPOPT}) solver on real datasets. 
For sake of fairness,  we only extend the first-order algorithmic framework (referred to as GS-ADMM) to tackle the $\ell_\infty$-DRSVM problem. In fact, the faster inner solver conjugate gradient with an active set method can only tackle the $\ell_\infty$ case in \cite{li2019first}.  The implementation details to reproduce all numerical results in this section are given in the Appendix. Our code is available at \url{https://github.com/gerrili1996/Incremental_DRSVM}.

\subsection{Synthetic data: Different regularity conditions and their step size schemes}
%We investigate various regularity conditions of the problem \eqref{eq:drsvm_our} and their step size schemes selection on synthetic data. 
Our setup for the synthetic experiments is as follows.  
First, we generate the learning parameter $w^*$ and feature vectors $\{{x}_i\}_{i=1}^n$ independently and identically (i.i.d) from the standard normal distribution $ \mathcal{N}(0, I_d)$ and the noisy measurements $\{\xi_i\}_{i=1}^n$ i.i.d from $ \mathcal{N}(0, \sigma^2 I_d)$ (e.g., $\sigma = 0.5$). Then, we compute the ground-truth labels $\{{y}_i\}_{i=1}^n$ by ${y}_i = \text{sign}(\langle w^*, x_i\rangle + \xi_i )$. 
Here, the model parameters are $n=1000,d=100,\kappa =1,\epsilon =0.1$. All the algorithmic parameters of ISG and IPPA have been fine-tuned via grid search for optimal performance.
% in Fig.\ref{fig:syn}.  As an ablation study, we use the single-sample update for ISG too. 
Recall from Theorem~\ref{thm:conv} that for instances satisfying the sharpness condition, the smaller shrinking rate $\rho$ the algorithm can adopt, the faster its linear rate of convergence. The experiments results in Fig.~\ref{fig:syn}(a--c) indicate that IPPA allows us to choose a more aggressive $\rho$ when compared with ISG over all instances satisfying the sharpness condition. 
A similar phenomenon has also been observed in previous works; see, e.g., \cite[Fig.~1]{li2019incremental}. Even for instances that do not satisfy the sharpness condition, IPPA performs better than ISG; see Fig.~\ref{fig:syn}(d). 

\begin{figure*}
	\centering
	\subfigure[]{\includegraphics[width=0.325\textwidth]{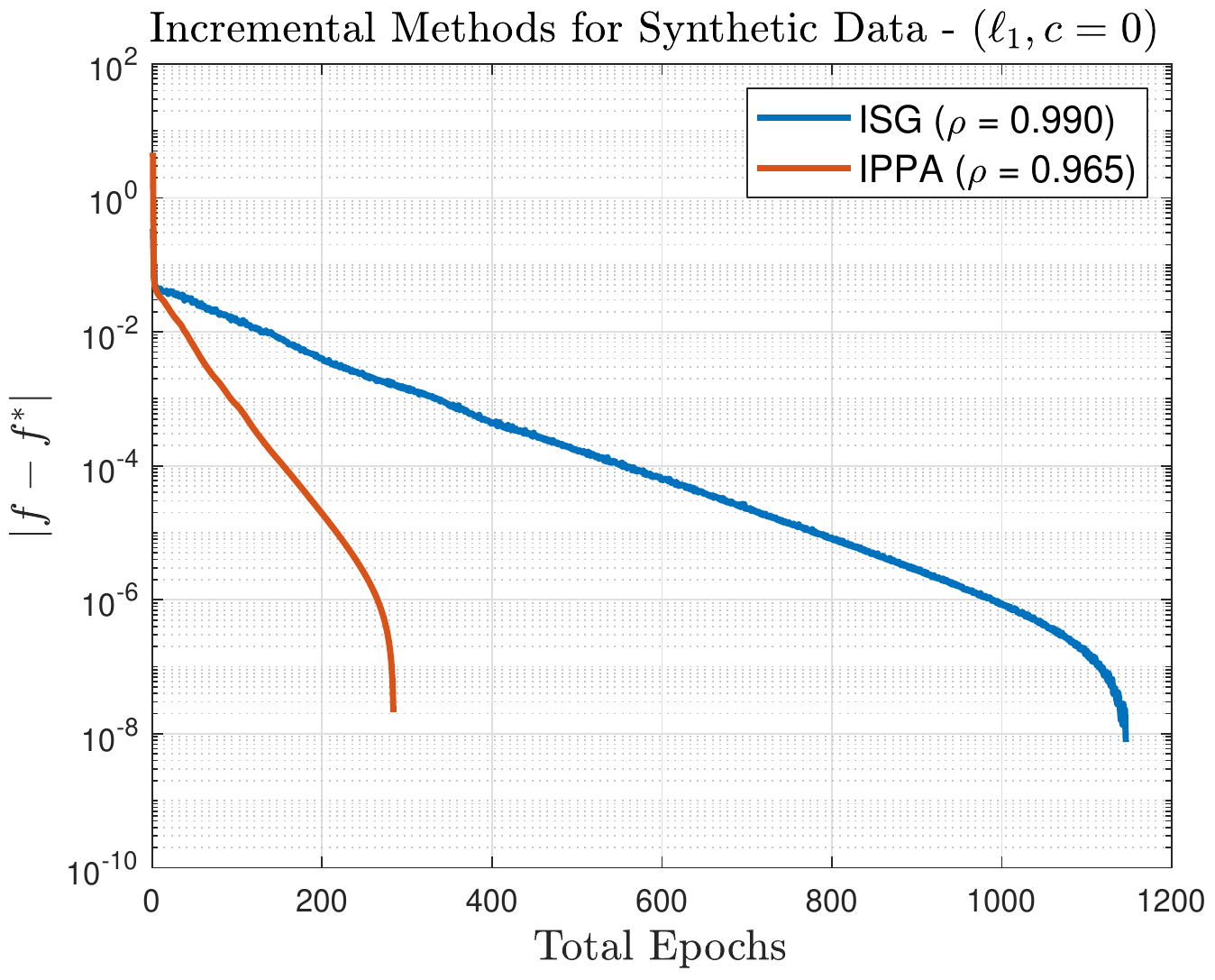}} 
	\subfigure[]{\includegraphics[width=0.325\textwidth]{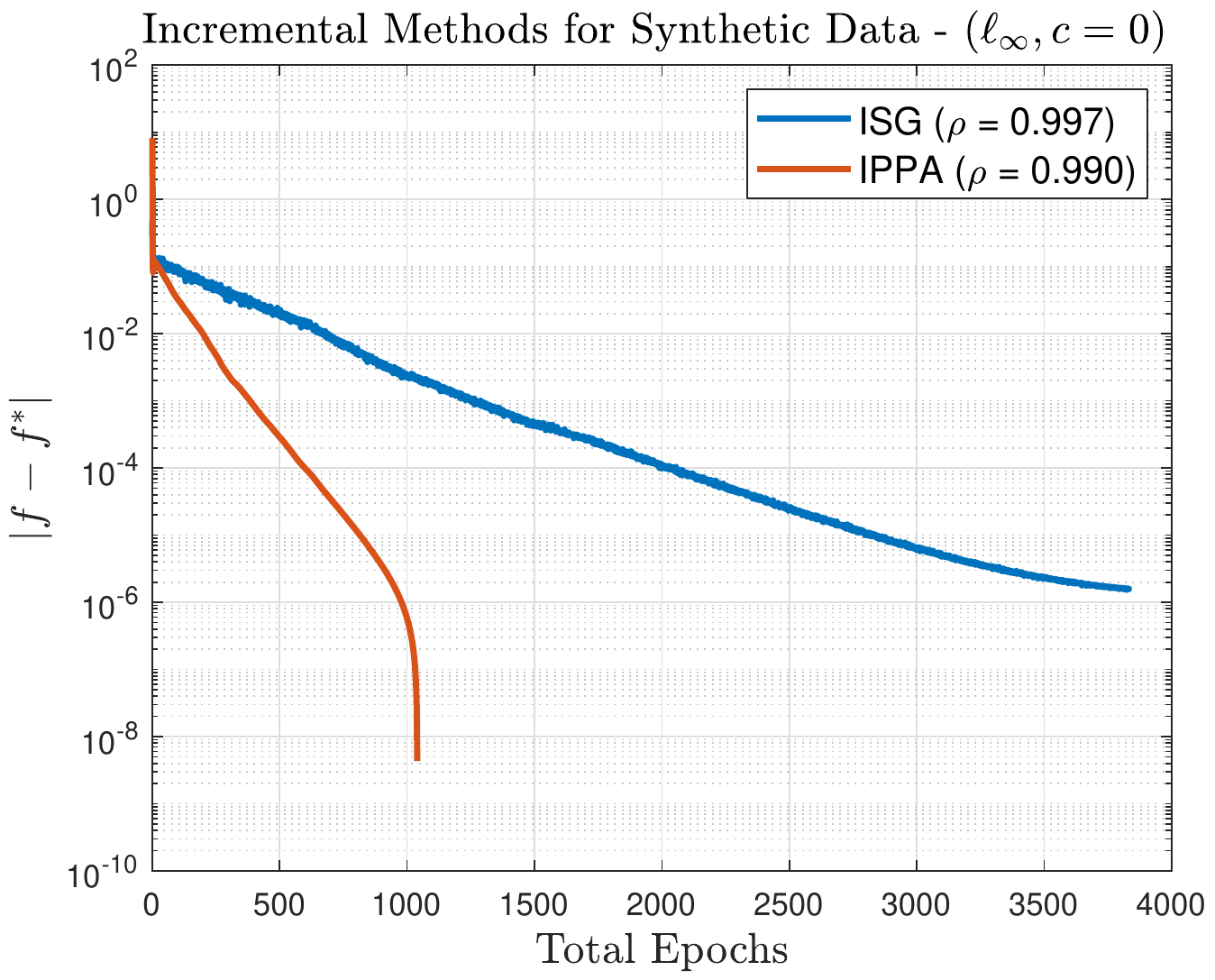}} 
	\subfigure[]{\includegraphics[width=0.325\textwidth]{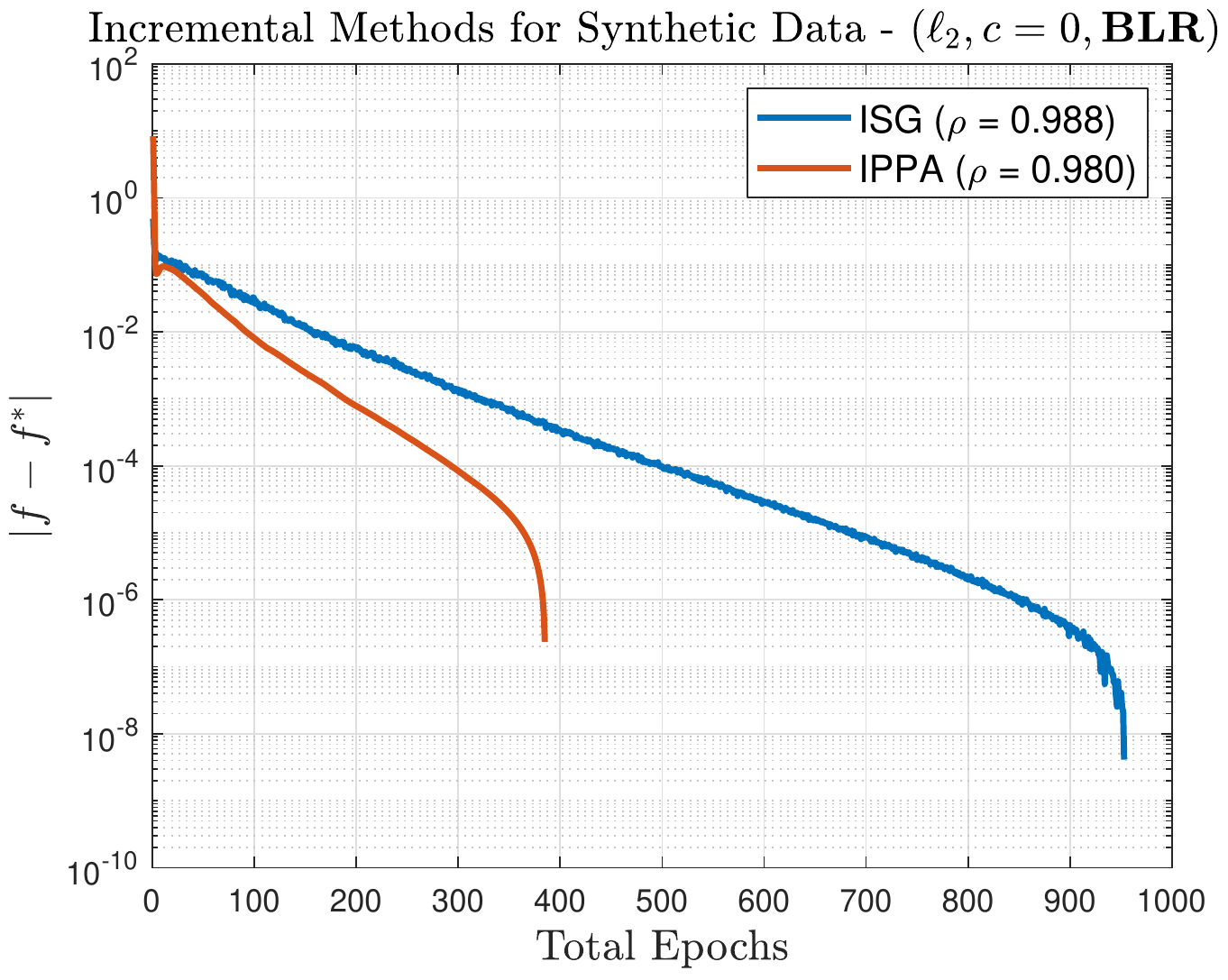}} \\
	\subfigure[]{\includegraphics[width=0.342\textwidth]{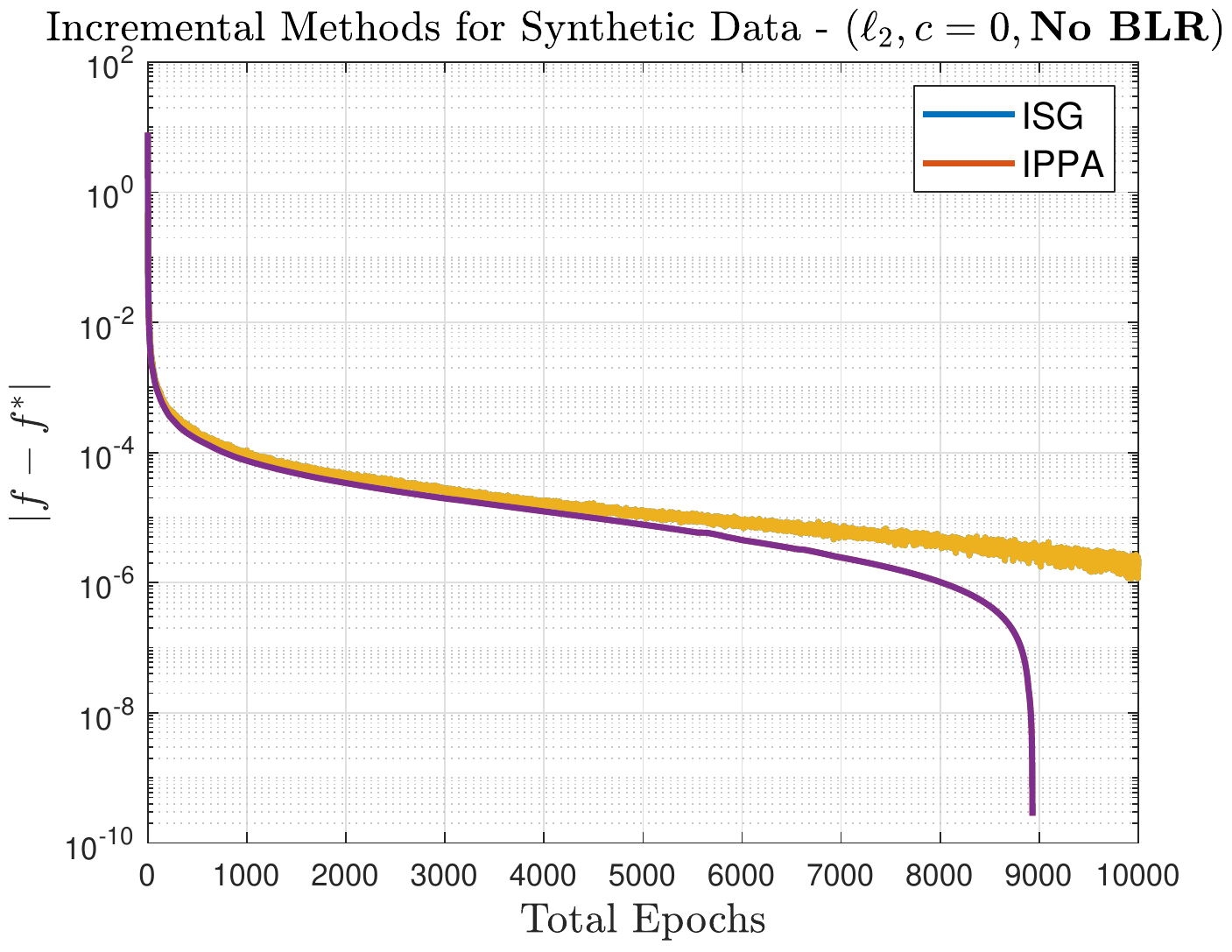}} 
	\subfigure[]{\includegraphics[width=0.322\textwidth]{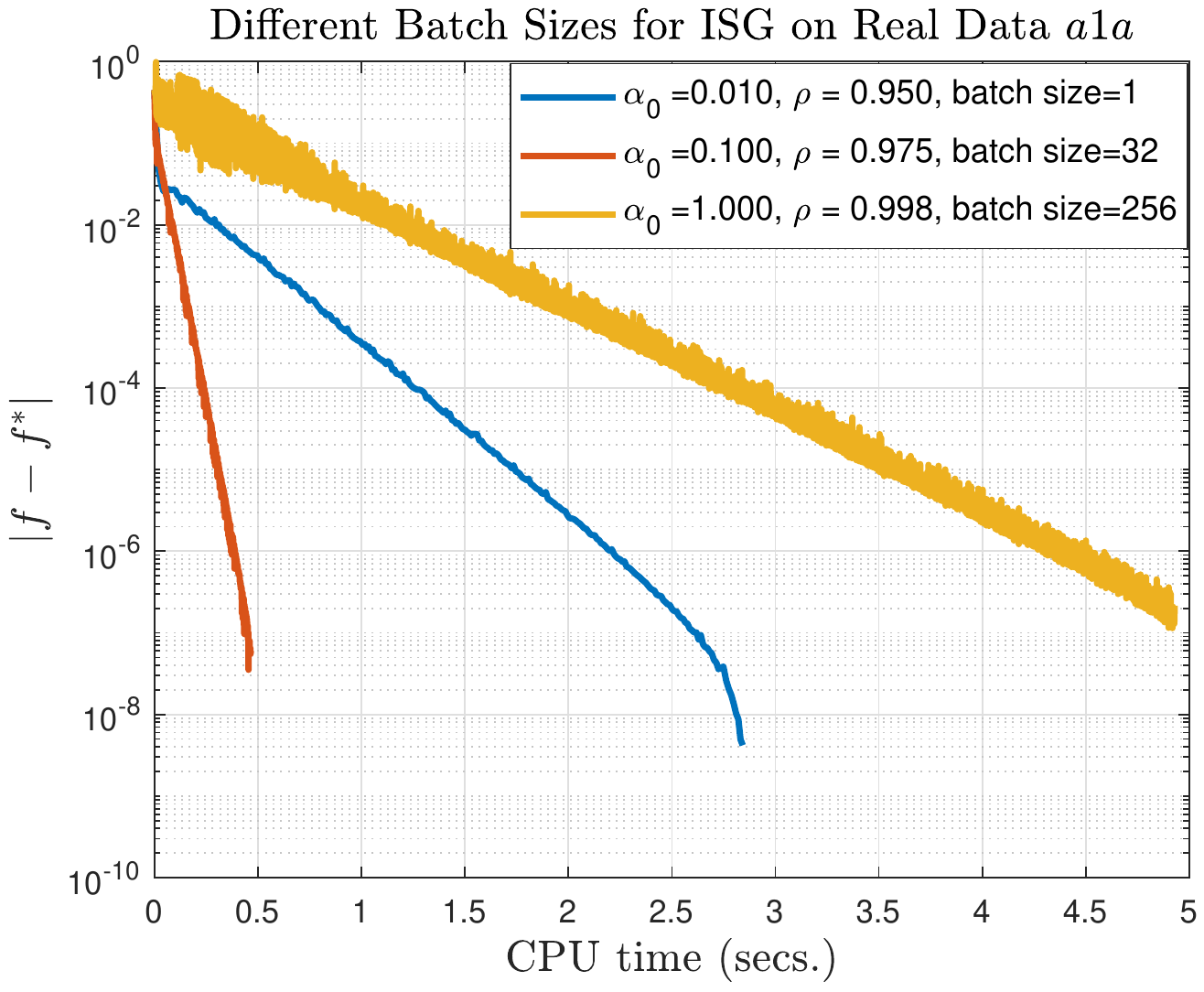}} 
	\subfigure[]{\includegraphics[width=0.322\textwidth]{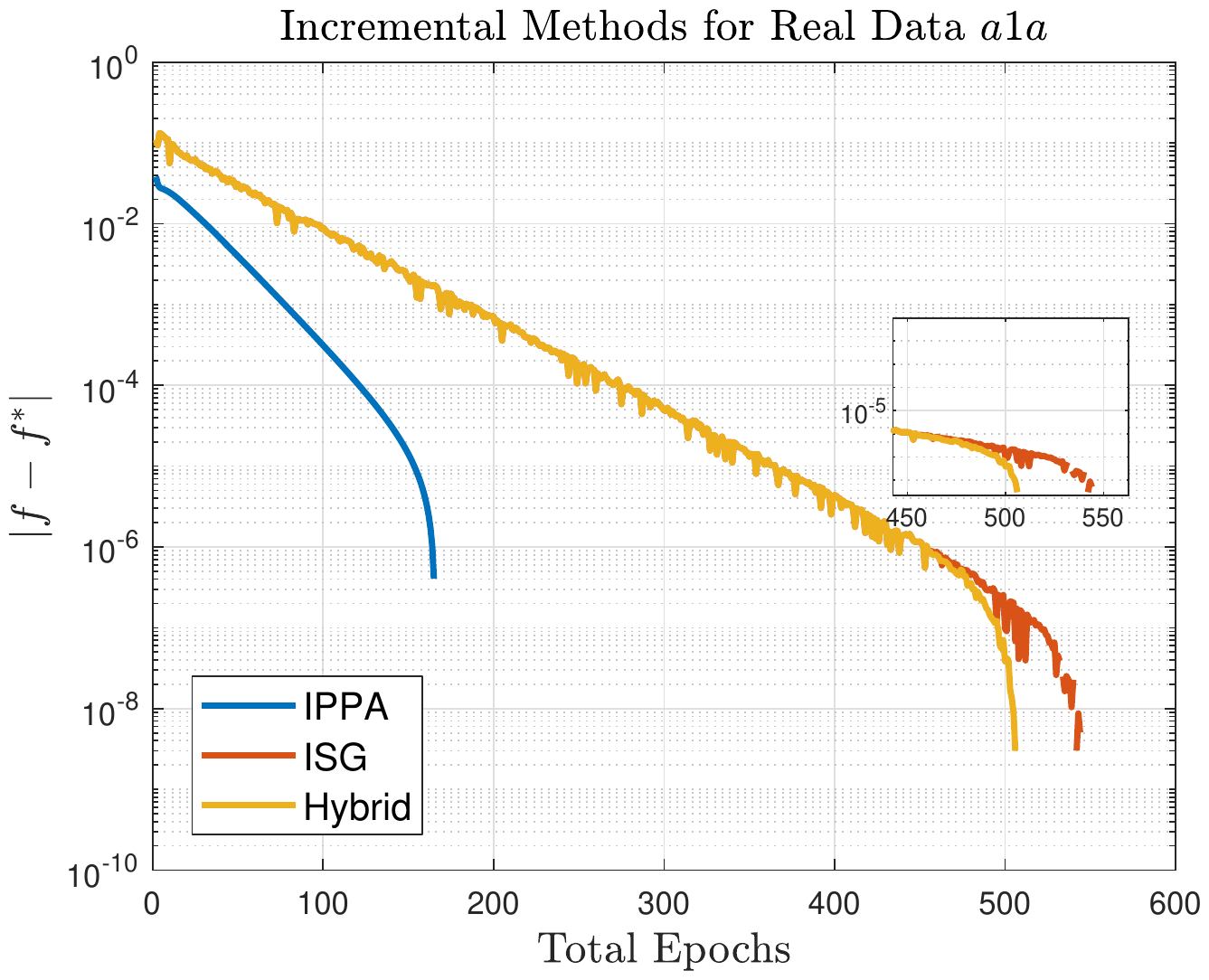}} 
	\caption{(a)--(d): Comparison between ISG and IPPA on both BLR and non-BLR instances generated from synthetic datasets. (e)--(f): Performance of ISG on different mini-batch sizes and performance of the hybrid algorithm on the \texttt{a1a} dataset.}
\label{fig:syn}
	\vspace{-0.2in}
\end{figure*}

%\begin{figure}
%	\centering
%	\includegraphics[width=\textwidth]{figure.jpg}	\caption{(a)--(d): Comparison between ISG and IPPA on both BLR and non-BLR instances generated from synthetic datasets. (e)--(f): Performance of ISG on different mini-batch sizes and performance of the hybrid algorithm on the \texttt{a1a} dataset.}
%	\label{fig:syn}
%\end{figure}
Nevertheless, IPPA can only handle one sample at a time. Thus, we are motivated to develop an approach that can combine the best features of both ISG and IPPA. Towards that end, observe from Fig.\ref{fig:syn}(e) that 
%There are two key points that we want to corroborate empirically. Indeed, the effect of mini-batch size for ISG is still underexplored. 
there is a tradeoff between the mini-batch size and the shrinking rate $\rho$, which means that there is an optimal mini-batch size for achieving the fastest convergence speed.
%the most aggressive shrinking rate $\rho$ is advanced with the growth of mini-batch size, 
Inspired by this, we propose to first apply the mini-batch ISG to obtain an initial point and then use IPPA in a local region around the optimal point to gain further speedup and get a more accurate solution. As shown in Fig.~\ref{fig:syn}(f), such a hybrid algorithm is effective, thus confirming our intuition.

 %To better showcase the comparison, we use the real dataset \texttt{a1a} to conduct experiments.  

\subsection{Efficiency of our incremental algorithms}
Next, we demonstrate the efficiency of our proposed methods on the real datasets \texttt{a1a-a9a,ijcnn1} downloaded from the LIBSVM\footnote{\url{https://www.csie.ntu.edu.tw/~cjlin/libsvmtools/datasets/binary.html}}. The results for $\ell_1$-DRSVM, which satisfies the sharpness condition, are shown in Table \ref{tab:1}. Apparently, IPPA is slower than mini-batch ISG (i.e., M-ISG) in general but can obtain more accurate solutions. More importantly, the hybrid algorithm, which combines the advantages of both M-ISG and IPPA, has an excellent performance and achieves a well-balanced tradeoff between accuracy and efficiency. All of them are much faster than YALMIP. The results for $\ell_2$-DRSVM are reported in Table \ref{tab:2}. As ISG is sensitive to hyper-parameters and has difficulty achieving the desired accuracy, we only present the results for IPPA. From the table, the superiority of IPPA over the solver is obvious.

\begin{table}[]
	\centering
	\caption{Wall-clock Time Comparison on UCI  Real Dataset: $\ell_1$-DRSVM, $c=0,\kappa=1,\epsilon=0.1$}
	\label{tab:1}
		\setlength{\tabcolsep}{4.5pt}
	\begin{tabular}{@{}ccccclcccc@{}}
		\toprule
		{\color{white}{\multirow{2}{*}{Dataset}}} & \multicolumn{4}{c}{Objective Value} &  & \multicolumn{4}{c}{Wall-clock time (sec)} \\ \cmidrule(lr){2-5} \cmidrule(l){7-10} 
		& M-ISG & IPPA & Hybrid & YALMIP &  & M-ISG & IPPA & Hybrid & YALMIP \\
		 \midrule
		a1a & \textbf{0.651090} & 0.651091 & \textbf{0.651090} & 0.651102 &  & \textbf{0.706} & 6.1242 & 1.560 & 12.221 \\
		a2a & \textbf{0.670640} & 0.670640 & \textbf{0.670640} & 0.670652 &  & \textbf{0.717} & 7.040 & 1.720 & 9.695 \\
		a3a & 0.662962 & 0.663093 & \textbf{0.662962} & 0.663060 &  & \textbf{1.800} & 21.242 & 3.740 & 11.854 \\
		a4a & 0.674274 & 0.674274 & \textbf{0.674273} & 0.674274 &  & \textbf{3.764} & 25.980 & 4.664 & 16.638 \\
		a5a & \textbf{0.660867} & \textbf{0.660867} & \textbf{0.660867} & 0.660869 &  & \textbf{2.026} & 24.752 & 24.752 & 24.207 \\
		a6a & \textbf{0.654189} & \textbf{0.654189} & \textbf{0.654189} & 0.654194 &  & \textbf{2.277} & 26.127 & 2.509 & 39.311 \\
		a7a & 0.656274 & 0.656274 & \textbf{0.656273} & 0.656411 &  & \textbf{2.528} & 33.094 & 2.799 & 60.046 \\
		a8a & 0.650036 & 0.650036 & \textbf{0.650035} & 0.650081 &  & \textbf{3.004} & 41.249 & 3.729 & 94.377 \\
		a9a & 0.642186 & 0.642186 & \textbf{0.642185} & 0.642596 &  & \textbf{2.285} & 35.554 & 3.063 & 155.980 \\ \bottomrule
	\end{tabular}
\end{table}

%
%\begin{table}[h]
%	\centering
%	\caption{Wall-clock Time Comparison on UCI Real Dataset: $\ell_2$-DRSVM, $c=0,\kappa=1,\epsilon=0.1$}
%	\label{tab:2}
%	\begin{tabular}{@{}ccccccc@{}}
%		\toprule
%		Dataset & Size & \multicolumn{2}{c}{\begin{tabular}[c]{@{}c@{}}Objective Value\\  IPPA | YALMIP\end{tabular}} & \multicolumn{2}{c}{\begin{tabular}[c]{@{}c@{}}Wall-clock time (sec)\\ IPPA | YALMIP\end{tabular}} & Regularity Condition \\ \midrule
%		a1a & (1605,123) & 0.6339472 & \textbf{0.6338819} & \textbf{5.517} & 8.557 & Not Known \\
%		a2a & (2265,123) & 0.6599856 & \textbf{0.6599099} & \textbf{9.355} & 11.989 & Not Known \\
%		a3a & (3185,123) & 0.6443777 & \textbf{0.6442762} & \textbf{7.096} & 15.335 & Not Known \\
%		a4a & (4781,123) & 0.6513987 & \textbf{0.6513899} & \textbf{14.162} & 23.122 & Not Known \\
%		a5a & (6414,123) & 0.6484421 & \textbf{0.6484147} & \textbf{10.515} & 32.663 & Not Known \\
%		a6a & (11220,123) & 0.6428831 & \textbf{0.6428806} & \textbf{15.195} & 67.695 & Not Known\\
%		a7a & (16100,123) & \textbf{{0.6459271}} & 0.6462302 & \textbf{{6.454}} & 118.740 & Not Known \\
%		a8a & (22696,123) & \textbf{{0.6441057}} & \textbf{0.6441057} & \textbf{{27.242}} & 161.000 & Not Known \\
%		a9a & (32561,123) & \textbf{{0.6389162}} & 0.6437767 & \textbf{{13.129}} & 215.387 & Sharpness \\
%		ijcnn & (49990,22) & \textbf{{0.4781876}} & 0.4781897 & \textbf{{20.567}} & 379.943 & Sharpness \\ \bottomrule
%	\end{tabular}
%\end{table}

\begin{table}[]
	\centering
	\caption{Wall-clock Time Comparison on UCI Real Dataset: $\ell_2$-DRSVM, $c=0,\kappa=1,\epsilon=0.1$}
	\label{tab:2}
	\begin{tabular}{@{}ccccccc@{}}
		\toprule
		{\color{white}{\multirow{2}{*}{Dataset}}} & \multicolumn{2}{c}{Objective Value} &  & \multicolumn{2}{c}{Wall-clock time (sec)} & \multirow{2}{*}{Regularity Condition} \\ \cmidrule(lr){2-6}
		& IPPA & YALMIP &  & IPPA & YALMIP &  \\ \midrule
		a1a & 0.6339472 & \textbf{0.6338819} &  & \textbf{5.517} & 8.557 & Not Known \\
		a2a & 0.6599856 & \textbf{0.6599099} &  & \textbf{9.355} & 11.989 & Not Known \\
		a3a & 0.6443777 & \textbf{0.6442762} &  & \textbf{7.096} & 15.335 & Not Known \\
		a4a & 0.6513987 & \textbf{0.6513899} &  & \textbf{14.162} & 23.122 & Not Known \\
		a5a & 0.6484421 & \textbf{0.6484147} &  & \textbf{10.515} & 32.663 & Not Known \\
		a6a & 0.6428831 & \textbf{0.6428806} &  & \textbf{15.195} & 67.695 & Not Known \\
		a7a & \textbf{0.6459271} & 0.6462302 &  & \textbf{6.454} & 118.740 & Not Known \\
		a8a & \textbf{0.6441057} & \textbf{0.6441057} &  & \textbf{27.242} & 161.000 & Not Known \\
		a9a & \textbf{0.6389162} & 0.6437767 &  & \textbf{13.129} & 215.387 & Sharpness \\
		ijcnn & \textbf{0.4781876} & 0.4781897 &  & \textbf{20.567} & 379.943 & Sharpness \\ \bottomrule
	\end{tabular}
\end{table}

To further demonstrate the efficiency of our proposed hybrid algorithm, we compare it with GS-ADMM~\cite{li2019first} and YALMIP on $\ell_\infty$-DRSVM, which again satisfies the sharpness condition. The results are shown in Table \ref{tab:3}. The overall performance of our hybrid method dominates both GS-ADMM and YALMIP. Due to space limitation, we only present the results for the case $q\in\{1,2,\infty\}$, $c=0$. More numerical results can be found in the Appendix. 
%\begin{table}[!htbp]
%		\centering
%	\caption{Wall-clock Time Comparison on UCI Real Dataset: $\ell_\infty$-DRSVM, $c=0,\kappa=1,\epsilon=0.1$}
%	\label{tab:3}
%	\begin{tabular}{@{}cccccccc@{}}
%		\toprule
%		Dataset & Hybrid & GS-ADMM & YALMIP & Dataset &Hybrid & GS-ADMM  & YALMIP \\ \midrule
%		a1a     & \textbf{4.789}      & 5.939          & 7.832  & a6a     & \textbf{8.273}      & 18.801  & 42.714  \\
%		a2a     & \textbf{5.098}      & 7.069          & 9.100  & a7a     & \textbf{6.115}      & 20.191  & 60.743  \\
%		a3a     & 16.252              & \textbf{9.638} & 11.375 & a8a     & \textbf{11.065}     & 27.867  & 99.355  \\
%		a4a     & \textbf{5.498}      & 10.446         & 17.542 & a9a     & \textbf{5.717}      & 31.278  & 172.070 \\
%		a5a     & \textbf{7.362}      & 13.993         & 22.969 & ijcnn   & \textbf{4.301}      & 12.415        & 319.379         \\ \bottomrule
%	\end{tabular}
%\end{table}

\begin{table}[!htbp]
	\centering
	\setlength{\tabcolsep}{4.5pt}
	\caption{Wall-clock Time Comparison on UCI Real Dataset: $\ell_\infty$-DRSVM, $c=0,\kappa=1,\epsilon=0.1$}
	\label{tab:3}
	\begin{tabular}{@{}ccccccccc@{}}
		\toprule
		{\color{white}{{Dataset}}} & Hybrid & GS-ADMM & YALMIP &  & {\color{white}{{Dataset}}}  & Hybrid & GS-ADMM & YALMIP \\ \midrule
		a1a & \textbf{4.789} & 5.939 & 7.832 &  & a6a & \textbf{8.273} & 8.273 & 42.714 \\
		a2a & \textbf{5.098} & 7.069 & 9.100 &  & a7a & \textbf{6.115} & 6.115 & 60.743 \\
		a3a & 16.252 & \textbf{9.638} & 11.375 &  & a8a & \textbf{11.065} & 11.065 & 99.355 \\
		a4a & \textbf{5.498} & 10.446 & 17.542 &  & a9a & \textbf{5.717} & 5.717 & 172.07 \\
		a5a & \textbf{7.363} & 13.993 & 22.969 &  & ijcnn & \textbf{4.301} & 4.301 & 319.379 \\ \bottomrule
	\end{tabular}
\end{table}
%\vspace{-0.5\baselineskip}
\section{Conclusion and Future Work}
%\vspace{-0.5\baselineskip}
In this paper, we developed two new and highly efficient epigraphical projection-based incremental algorithms to solve the Wasserstein DRSVM problem with $\ell_p$ norm-induced transport cost ($p\in\{1,2,\infty\}$) and established their convergence rates.
A natural future direction is to develop a mini-batch version of IPPA and extend our algorithms to the asynchronous decentralized parallel setting.
Inspired by our paper, it would also be interesting to develop some new incremental/stochastic algorithms to tackle more general Wasserstein DRO problems; see, e.g., problem (11) in~\cite{esfahani2018data}.

\paragraph{Acknowledgment} Caihua Chen is supported in part by the National Natural Science Foundation of China (NSFC) projects 71732003, 11871269 and in part by the Natural Science Foundation of Jiangsu Province project BK20181259. Anthony Man-Cho So is supported in part by the CUHK Research Sustainability of Major RGC Funding Schemes project 3133236.

\paragraph{Broader Impact}This work does not present any foreseeable societal consequence. A broader impact discussion is not applicable. 

\bibliographystyle{plain}
\bibliography{ref}

\newpage
\section*{Appendix}
This supplementary document is the appendix section of the paper titled ``\textbf{Fast Epigraphical Projection-based Incremental Algorithms for Wasserstein Distributionally Robust Support Vector Machine}''. It is organized as follows. In Section A, we give the details of the algorithms for solving the subproblems (i.e., $\ell_q$ norm epigraphical projection and single-sample proximal point update \eqref{eq:IPPA}). In Section B, we prove the results in the section ``Convergence Rate Analysis of Incremental Algorithms''. In Section C, we describe how to extend the algorithm in~\cite{li2019first} (i.e., GS-ADMM) to tackle our $\ell_\infty$-DRSVM problem. Subsequently, we provide additional experimental results to further demonstrate the effectiveness of our proposed method.  

\subsection*{A: Algorithmic ingredients in ISG and IPPA}
To begin, we provide a summary in Table \ref{tab:sub}, which aims to help the reader find the related algorithmic details as quickly as possible. 
\begin{table}[H]
	\centering
	\caption{Summary of all ingredients in ISG and IPPA  }
	\label{tab:sub}
	\begin{tabular}{ccc}
		\toprule
		Cases & ISG epigraphical projection~\eqref{eq:ISG} & IPPA single sample update~\eqref{eq:IPPA}\\
		\midrule
		\multicolumn{1}{c}{\multirow{2}{*}{$\ell_2$}} & \multicolumn{1}{c}{\multirow{2}{*}{closed-form; see Prop. \ref{prop:epi2}}} & exhaust all seven cases; see Table \ref{tb:l2} \\
		\multicolumn{1}{c}{} & \multicolumn{1}{c}{} & analytic form for subcases; see Prop. \ref{prop:l2sub},\ref{prop:l2sub2} \\
		$\ell_1$ & quick-select algorithm in linear time & exhaust all five cases; see Alg. \ref{algo:l1sub} \\ 
		$\ell_\infty$ & Moreau's decomposition based on $\ell_1$ case & modified secant alg. \ref{algo:MSA} \\ \bottomrule
	\end{tabular}
\end{table}

%\multirow{2}{*}{\begin{tabular}[c]{@{}l@{}}exhaust all five cases; see Alg. \ref{algo:l1sub} \\ modified secant alg. \ref{algo:MSA};\end{tabular}}
\paragraph{Subproblems for $q=2$.} It is well known that the $\ell_2$ norm epigraphic projection has a closed-form formula, which is given as follows:
\begin{proposition}[Adopted from~{\cite[Theorem 3.3.6]{bauschke1996projection}}]  Let $L^d_2 = \{(x,s)\in \mathbb{R}^d \times \mathbb{R}: \|x\|_2 \leq s \}$. For any $(x,t)\in \mathbb{R}^d \times \mathbb{R}$, we have
	\begin{equation}
	\proj_{L^d_2}(x,s)=\left\{
	\begin{array}{ccl}
	\left( \frac{\|x\|_2+s}{2\|x\|_2},\frac{\|x\|_2+s}{2} \right) & & {\|x\|_2 \ge |s|}, \\
	(0,0)& & {s<\|x\|_2<-s}, \\
	(x,s) & & {\|x\|_2\leq s}.
	\end{array} \right.
	\end{equation}
	\label{prop:epi2}
\end{proposition}
Recall that the  $\ell_2$ single-sample proximal point subproblem~\eqref{eq:l2} takes the form 
\begin{equation*}
	\min\limits_{w,\lambda}  \underbrace{\max \left\{1-w^Tz_i, 1+w^Tz_i-\lambda\kappa,0\right\}}_{h_i(w,\lambda)} + \frac{1}{2\alpha}(\|w-\bar{w}\|_2^2 + (\lambda-\bar{\lambda})^2),~\ \text{s.t.} \ \|w\|_2 \leq \lambda.
\end{equation*}
We start with $\Gamma =1, j=1$. Problem \eqref{eq:case1} can be written as
\begin{equation*}
\min\limits_{w,\lambda} \frac{1}{2\alpha}(\|w-\bar{w}-\alpha z_i\|_2^2 + (\lambda-\bar{\lambda})^2),~\ \text{s.t.} \ \|w\|_2 \leq \lambda,
\end{equation*}
whose optimal solution is given by $(\hat{w},\hat{\lambda}) = \proj_{L^d_2}(\bar{w}+\alpha z_i,\bar{\lambda})$. We further check whether $(\hat{w},\hat{\lambda})$ satisfies the optimality condition of problem~\eqref{eq:l2}; i.e., $\{(w,\lambda):h_{i,1}(w,\lambda)>h_{i,2}(w,\lambda), \, h_{i,1}(w,\lambda)>h_{i,3}(w,\lambda)\} = \{(w,\lambda):w^Tz<\min(\frac{\lambda\kappa}{2},1)\}$. The other two cases (i.e., $\Gamma=1, j=2,3$) follow the same procedure. Then, we proceed to consider $\Gamma=2$   (e.g., $(j,j') = (1,2)$). Problem \eqref{eq:case2} is equivalent to 
\begin{equation*}
\min_{w,\lambda} \frac{1}{2\alpha}(\|w-\bar{w}-\alpha z_i\|_2^2 + (\lambda-\bar{\lambda})^2),~\ \text{s.t.}~w^Tz_i = \frac{\kappa}{2}\lambda,~\|w\|_2 \leq \lambda.
\end{equation*} 
We now derive an analytic solution for its prototypical form~\eqref{eq:pp2sub_1} in Proposition \ref{prop:l2sub}, which also covers the other two cases  (i.e., $\Gamma=2, (j,j') = (1,3),(2,3)$).  
\begin{proposition} \label{prop:l2sub}
Given $\bar{w} \in \mathbb{R}^d$ and $\bar{\lambda},a,b \in \mathbb{R}$, consider the following optimization problem:
	\begin{equation}
	\begin{aligned}
 	&	\min_{w,\lambda}
	& & \frac{1}{2}\|w-\bar{w}\|_2^2 +\frac{1}{2}(\lambda -  \bar{\lambda})^2 \\
	&\,\,\, {\rm s.t.}
	& & w^Tz_i = a\lambda+b \leftarrow{\mu_1},\\
	& && \|w\|_2 \leq \lambda \leftarrow{\mu_2},
	\end{aligned} 
	\label{eq:pp2sub_1} 
	\end{equation}
where $\mu_1$ and $\mu_2$ are the associated dual multipliers.  Then, the optimal solution $(w^*,\lambda^*)$ to \eqref{eq:pp2sub_1} is 
\[
	\PPA(\bar{w},\bar{\lambda},a,b) \triangleq \left( \frac{\bar{w} -\mu_1^* z_i}{1+2\mu_2^*}, \frac{\bar{\lambda} + a\mu_1^*}{1-2\mu_2^*} \right),
\]
 where $\mu_1^*$ and $\mu_2^*$ are the optimal dual multipliers. In particular, we have
	\[\mu_1^* = \frac{(1-2\mu_2^*)\bar{w}^Tz_i-a(1+2\mu_2^*)\bar{\lambda} - b(1-2\mu_2^*)(1+2\mu_2^*)}{(1+2\mu_2^*)a^2 + (1-2\mu_2^*)\|z_i\|_2^2},\]
\[
 \mu_2^* = \left\{
	\begin{aligned}
	&0, \qquad\text{if}~\, \|\bar{w} - \mu_1^* z_i\|_2\leq \bar{\lambda} + a\mu_1^*,\\
	& \hat{\mu}_2, \quad\,\,\text{otherwise}
	\end{aligned}
	\right.
\]
and $\hat{\mu}_2$ is the root of the following \textbf{quartic equation} satisfying $\hat{\mu}_2 > 0$ and $\lambda ^* \ge 0$:
	\[p_1 \mu_2^4 + p_2 \mu_2^3 + p_3 \mu_2^2 + p_4 \mu_2 + p_5 =0.\]
	Here, $A =\|\bar{w}\|_2^2$, $B = \|z_i\|_2^2$, $C = \bar{w}^Tz_i$, and
	\begin{equation}\label{eq:quartic}
	\begin{aligned}
	p_1 = & a^2b^2- Bb^2, \\
	p_2 = &4a^2b^2, \\
	p_3 = &-4Bab\bar{\lambda}+2BCa\bar{\lambda}- 2Ca^3\bar{\lambda}- 4Ca^2b+2ABa^2+6a^2b^2 \\
	&+B^2\bar{\lambda}^2+2Bb^2+ BC^2- Ba^2\bar{\lambda}^2- C^2a^2- Aa^4 - AB^2, \\
	p_4 = & - 8Bab\bar{\lambda} + 4BCa\bar{\lambda} - 2Ba^2\bar{\lambda}^2 - 4Ca^3\bar{\lambda} - 8Ca^2b - 2Aa^4 \\
	& - 2BC^2 + 2AB^2 + 4a^2b^2 + 2B^2\bar{\lambda}^2 + 2C^2a^2, \\
	p_5 = & - 4Bab\bar{\lambda} + 2BCa\bar{\lambda} - 2Ca^3\bar{\lambda} - 4Ca^2b - 2ABa^2 + a^2b^2 \\
	& + B^2\bar{\lambda}^2 + 3C^2a^2 + BC^2 - Ba^2\bar{\lambda}^2 - Bb^2 - Aa^4 - AB^2.
	\end{aligned}
	\end{equation}
\end{proposition}
\begin{proof}
%	The Lagrangian function associated with~\eqref{eq:pp2sub_1} is  
%	\begin{equation}
%	L(w,\lambda, \mu_1, \mu_2) = \frac{1}{2}\|w-\bar{w}\|_2^2 +\frac{1}{2}(\lambda -\bar{\lambda})_2^2 + \mu_1 ( w^Tz_i  -a\lambda - b) + \mu_2 (\|w\|_2^2 - \lambda^2). 
%	\end{equation}
The Karush-Kuhn-Tucker (KKT) conditions of \eqref{eq:pp2sub_1} are given by 
\begin{equation}
\label{eq:KKT_sub}
% \text{Stationary Condition} \longrightarrow 
\left\{ 
\begin{aligned}
& w^* -\bar{w} + \mu_1^* z_i + 2\mu_2^* w^*  = 0,\\
& \lambda^* -\bar{\lambda} - a\mu_1^* - 2 \mu_2^*\lambda^* = 0,\\
& {w^*} ^Tz_i = a\lambda^* + b,\\
& \|w^* \|_2 \leq \lambda^*,\\
& \mu_2^*(\|w^* \|_2^2-{\lambda^*}^2) = 0, \\
& \mu_2^* \ge 0 . 
\end{aligned}
\right.
\end{equation}
Based on \eqref{eq:KKT_sub}, we have 	
\begin{align*}
& (1+2\mu_2^*){w^*}^Tz_i-\bar{w}^Tz_i + \mu_1^* \|z_i\|_2^2  = 0 \quad\Rightarrow\quad {w^*} ^Tz_i = \frac{\bar{w}^Tz_i -\mu_1^* \|z_i\|_2^2}{1+2\mu_2^*}, \\
& (1-2\mu_2^*)\lambda^* -\bar{\lambda} - a\mu_1^*  = 0  \quad\Rightarrow\quad (1-2\mu_2^*)(a\lambda^*+b) = a\bar{\lambda}+a^2\mu_1^*+b(1-2\mu_2^*).
\end{align*}
Plugging in ${w^* }^Tz_i = a\lambda + b$ yields
\begin{equation*}
\begin{aligned}
& (1-2\mu_2^*)\frac{\bar{w}^Tz_i -\mu_1^* \|z_i\|_2^2}{1+2\mu_2^*} = a\bar{\lambda}+a^2\mu_1^*+b(1-2\mu_2^*) \\
\Rightarrow\quad & (1-2\mu_2^*)(\bar{w}^Tz_i -\mu_1^* \|z_i\|_2^2) = a\bar{\lambda}(1+2\mu_2^*)+a^2\mu_1^*(1+2\mu_2^*)) + b(1-2\mu_2^*)(1+2\mu_2^*).
\end{aligned}
\end{equation*} 
Then, 
\begin{equation}
	\mu_1^* = \frac{(1-2\mu_2^*)\bar{w}^Tz_i-a(1+2\mu_2^*)\bar{\lambda} - b(1-2\mu_2^*)(1+2\mu_2^*)}{(1+2\mu_2^*)a^2 + (1-2\mu_2^*)\|z_i\|_2^2}.
	\label{eq:u1u2}
\end{equation}
To handle the complementary slackness condition $\mu_2^*(\|w^*\|_2^2-{\lambda^*}^2)=0$, we consider the following two cases:
	\begin{itemize}
		\item Case 1: If $\mu_2^*  = 0 $, then we have 
		$ \mu_1^* = \frac{\bar{w}^Tz_i-a\bar{\lambda} - b}{a^2 + \|z_i\|_2^2}$ and hence 
		\[ w^*  = \bar{w} - \mu_1^* z_i, \quad \lambda^* = \bar{\lambda} + a\mu_1^*.\]
		If $\|w^*\|_2 \leq \lambda^*$ does not hold, we go to Case 2.
		\item Case 2: If  $\mu_2^* > 0 $, then by incorporating $\|w^*\|_2 = \lambda^*$ into \eqref{eq:u1u2}, we obtain the quartic equation
		\[p_1 \mu_2^4 + p_2 \mu_2^3 + p_3 \mu_2^2 + p_4 \mu_2 + p_5 =0, \]
		whose coefficients $p_1,\ldots,p_5$ are given in \eqref{eq:quartic}. Finally, $\mu_2^*$ is in effect the root of this quartic equation, which satisfies $\mu_2^* > 0$ and $\lambda ^* \ge 0$. The optimal solution $(w^*,\lambda^*)$ is then given by
		\[w^* = \frac{\bar{w}-\mu_1^*z_i}{1+2\mu_2^*}, \quad \lambda^* =\frac{\bar{\lambda} + a\mu_1^*}{1-2\mu_2^*}. \]
	\end{itemize}
\end{proof}
\begin{Remark}
	The KKT conditions are necessary and sufficient for optimality for problem~\eqref{eq:pp2sub_1}.  If there does not exist a KKT point $(w^*,\lambda^*, \mu_1^*, \mu_2^*)$ (i.e., no nonnegative roots for the quartic function), then the case $\Gamma=2$ is not optimal for \eqref{eq:l2} and we proceed to other cases. For practical implementation, we apply the built-in function \texttt{roots([p1,p2,p3,p4,p5])} in MATLAB to get the roots of the quartic equation. 
\end{Remark}
Similarly, we check the corresponding optimality condition
$\{(w,\lambda):h_{i,1}(w,\lambda) = h_{i,2}(w,\lambda) > h_{i,3}(w,\lambda)\} = \{(w,\lambda): \lambda\kappa-1<w^Tz_i<1, \lambda\kappa<2\}$. The other two cases follow the same procedure. Lastly, we proceed to the case $\Gamma=3$ and problem \eqref{eq:case2} in effect admits a closed-form update.
\begin{proposition}	\label{prop:l2sub2}
	Given $\bar{w} \in \mathbb{R}^d$ and $b,\lambda\in\mathbb{R}$, consider the following optimization problem:
\begin{equation}
	\begin{aligned}
	&	\min_{w}
	& & \frac{1}{2}\|w-\bar{w}\|_2^2 \\
	&\,\,\, {\rm s.t.}
	& & w^Tz_i = b \leftarrow{\alpha},\\
	& && \|w\|_2 \leq \lambda \leftarrow{\beta},
	\end{aligned} 
	\label{eq:pp2sub_2} 
\end{equation}
where $\alpha$ and $\beta$ are the associated dual multipliers.  Then, the optimal solution $w^*$ to \eqref{eq:pp2sub_2} is 
	\[
	\mathcal{O}_{BH}(b,\lambda,\bar{w}) \triangleq \left\{
	\begin{array}{c@{\quad}l}
	A, & \text{if}\ \|A\|_2 \leq \lambda, \\
	\displaystyle \frac{1}{2\beta^*+1}\left\{ A+\frac{2b\beta^*}{\|z_i\|_2^2}z_i\right\}, & \text{otherwise},
	\end{array}
	\right.
	\]
	where $A =\bar{w}-\frac{\bar{w}^Tz_i-b}{\|z_i\|_2^2}z_i$, $B = \frac{2b}{\|z_i\|_2^2}z_i$, and $\beta^*$ is the positive root of the following \textbf{quadratic} equation:
	\[(4\lambda^2-\|B\|_2^2)\beta^2+(4\lambda^2-2A^TB)\beta+(\lambda^2-\|A\|_2^2)=0.\]
\end{proposition}

\begin{proof}	
%	The so-called Lagrangian function is  
%	\[L(w,\alpha,\beta) = \frac{1}{2}\|w-\bar{w}\|_2^2 + \alpha(w^Tz_i-b)+\beta(\|w\|_2^2-\lambda^2).\]
The KKT conditions of \eqref{eq:pp2sub_2} are given by
	\begin{equation}
	\left\{
	\begin{aligned}
	&w^*-\bar{w}+\alpha^* z_i +2\beta^* w^* =0,\\
	&{w^*}^Tz_i-b=0,\\
	&\|w^*\|_2 \leq \lambda,\\
	&\beta^*(\|w^*\|_2^2 -\lambda^2)=0,\\
	& \beta^* \ge 0.
	\end{aligned}
	\right.
	\label{eq:KKT_sub2}
	\end{equation}
Here, $\alpha^*$ and $\beta^*$ are the optimal dual multipliers. On top of \eqref{eq:KKT_sub2}, we have 
\[(1+2\beta^*)w^*-\bar{w}+\alpha^* z_i  =0 \quad\Rightarrow\quad (1+2\beta^*){w^*}^Tz_i-\bar{w}^Tz_i+\alpha^* \|z_i\|_2^2  =0,\]
\[(1+2\beta^*)b-\bar{w}^Tz_i+\alpha^* \|z_i\|_2^2  =0 \quad\Rightarrow\quad \alpha^* = \frac{\bar{w}^Tz_i - (1+2\beta^*)b}{\|z_i\|_2^2 }.\]
Plugging in $w^* = \frac{\bar{w}-\alpha^* z_i}{1+2\beta^*}$ gives
\[ w^* = \frac{1}{2\beta^*+1}\left\{ \bar{w}-\frac{\bar{w}^Tz-(1+2\beta^*)b}{\|z_i\|_2^2}z_i\right\}.
	\] 
Similarly, to handle the complementary slackness condition, we consider the case $\beta^* =0$. For this case, we check whether the condition $\|A\|_2 \leq \lambda$ holds. Otherwise, $\beta^*>0$ and $\|w^*\|_2^2 = \lambda^2$. This is equivalent to finding the positive root of the quadratic function
	\[(4\lambda^2-\|B\|_2^2){\beta^*}^2+(4\lambda^2-2A^TB)\beta^*+(\lambda^2-\|A\|_2^2)=0.\]
	
The geometric interpretation of this case is illustrated in Fig. \ref{fig:geo}. Problem \eqref{eq:pp2sub_2} seeks to find the projection onto the intersection of the Euclidean ball $\|w\|_2\leq \lambda $ and the hyperplane $w^Tz_i =b$. Observe that the projection of $\bar{w}$ onto the hyperplane $w^Tz_i = b$ is given by $A =\bar{w}-\frac{\bar{w}^Tz_i-b}{\|z_i\|_2^2}z_i$. If $A\in\{w:\|w\|_2\leq \lambda\}$ (i.e., red case), then $A$ is the optimal solution to \eqref{eq:pp2sub_2}. Otherwise, we aim to find the point on the sphere (i.e.,$\|w^*\|_2 =\lambda$) that is closer to $A$ (i.e., blue case). 
\begin{figure}[H]
	\centering
	\tikzset{every picture/.style={line width=0.75pt}} %set default line width to 0.75pt        
	
	\begin{tikzpicture}[x=0.75pt,y=0.75pt,yscale=-1,xscale=1]
	
	%Straight Lines [id:da14250501350486222] 
	\draw    (126.5,168) -- (417.5,168) ;
	%Shape: Circle [id:dp07026592919753427] 
	\draw   (219,196.6) .. controls (219,161.59) and (247.39,133.2) .. (282.4,133.2) .. controls (317.41,133.2) and (345.8,161.59) .. (345.8,196.6) .. controls (345.8,231.61) and (317.41,260) .. (282.4,260) .. controls (247.39,260) and (219,231.61) .. (219,196.6) -- cycle ;
	%Shape: Circle [id:dp5869051980420628] 
	\draw  [color={rgb, 255:red, 74; green, 144; blue, 226 }  ,draw opacity=1 ][fill={rgb, 255:red, 74; green, 144; blue, 226 }  ,fill opacity=1 ] (336.8,168.1) .. controls (336.8,166.94) and (337.74,166) .. (338.9,166) .. controls (340.06,166) and (341,166.94) .. (341,168.1) .. controls (341,169.26) and (340.06,170.2) .. (338.9,170.2) .. controls (337.74,170.2) and (336.8,169.26) .. (336.8,168.1) -- cycle ;
	%Shape: Circle [id:dp22509559912583166] 
	\draw  [color={rgb, 255:red, 74; green, 144; blue, 226 }  ,draw opacity=1 ][fill={rgb, 255:red, 74; green, 144; blue, 226 }  ,fill opacity=1 ] (223.8,168.1) .. controls (223.8,166.94) and (224.74,166) .. (225.9,166) .. controls (227.06,166) and (228,166.94) .. (228,168.1) .. controls (228,169.26) and (227.06,170.2) .. (225.9,170.2) .. controls (224.74,170.2) and (223.8,169.26) .. (223.8,168.1) -- cycle ;
	%Shape: Circle [id:dp7370501845429356] 
	\draw  [color={rgb, 255:red, 0; green, 0; blue, 0 }  ,draw opacity=1 ][fill={rgb, 255:red, 0; green, 0; blue, 0 }  ,fill opacity=1 ] (283.5,195.55) .. controls (283.5,194.97) and (283.97,194.5) .. (284.55,194.5) .. controls (285.13,194.5) and (285.6,194.97) .. (285.6,195.55) .. controls (285.6,196.13) and (285.13,196.6) .. (284.55,196.6) .. controls (283.97,196.6) and (283.5,196.13) .. (283.5,195.55) -- cycle ;
	%Straight Lines [id:da653055683764973] 
	\draw    (284.55,195.5) -- (326.34,237.82) ;
	\draw [shift={(328.45,239.95)}, rotate = 225.36] [fill={rgb, 255:red, 0; green, 0; blue, 0 }  ][line width=0.08]  [draw opacity=0] (8.93,-4.29) -- (0,0) -- (8.93,4.29) -- cycle    ;
	%Straight Lines [id:da05164087876916046] 
	\draw [color={rgb, 255:red, 208; green, 2; blue, 27 }  ,draw opacity=1 ] [dash pattern={on 4.5pt off 4.5pt}]  (297.5,78) -- (297.74,166) ;
	\draw [shift={(297.75,168)}, rotate = 269.84000000000003] [color={rgb, 255:red, 208; green, 2; blue, 27 }  ,draw opacity=1 ][line width=0.75]    (10.93,-3.29) .. controls (6.95,-1.4) and (3.31,-0.3) .. (0,0) .. controls (3.31,0.3) and (6.95,1.4) .. (10.93,3.29)   ;
	%Shape: Circle [id:dp33284366090736195] 
	\draw  [color={rgb, 255:red, 208; green, 2; blue, 27 }  ,draw opacity=1 ][fill={rgb, 255:red, 208; green, 2; blue, 27 }  ,fill opacity=1 ] (295.65,168) .. controls (295.65,166.84) and (296.59,165.9) .. (297.75,165.9) .. controls (298.91,165.9) and (299.85,166.84) .. (299.85,168) .. controls (299.85,169.16) and (298.91,170.1) .. (297.75,170.1) .. controls (296.59,170.1) and (295.65,169.16) .. (295.65,168) -- cycle ;
	%Straight Lines [id:da8765511179387935] 
	\draw [color={rgb, 255:red, 74; green, 144; blue, 226 }  ,draw opacity=1 ] [dash pattern={on 4.5pt off 4.5pt}]  (176.5,77) -- (176.74,165) ;
	\draw [shift={(176.75,167)}, rotate = 269.84000000000003] [color={rgb, 255:red, 74; green, 144; blue, 226 }  ,draw opacity=1 ][line width=0.75]    (10.93,-3.29) .. controls (6.95,-1.4) and (3.31,-0.3) .. (0,0) .. controls (3.31,0.3) and (6.95,1.4) .. (10.93,3.29)   ;
	%Shape: Circle [id:dp028675400665672646] 
	\draw  [color={rgb, 255:red, 74; green, 144; blue, 226 }  ,draw opacity=1 ][fill={rgb, 255:red, 74; green, 144; blue, 226 }  ,fill opacity=1 ] (174.65,167) .. controls (174.65,165.84) and (175.59,164.9) .. (176.75,164.9) .. controls (177.91,164.9) and (178.85,165.84) .. (178.85,167) .. controls (178.85,168.16) and (177.91,169.1) .. (176.75,169.1) .. controls (175.59,169.1) and (174.65,168.16) .. (174.65,167) -- cycle ;
	
	% Text Node
	\draw (378,174) node [anchor=north west][inner sep=0.75pt]  [font=\footnotesize]  {$w^{T} z_{i} \ =\ b\ $};
	% Text Node
	\draw (308,197) node [anchor=north west][inner sep=0.75pt]  [font=\footnotesize]  {$\lambda $};
	% Text Node
	\draw (269,184) node [anchor=north west][inner sep=0.75pt]  [font=\footnotesize]  {$o$};
	% Text Node
	\draw (290,56) node [anchor=north west][inner sep=0.75pt]  [font=\footnotesize,color={rgb, 255:red, 208; green, 2; blue, 27 }  ,opacity=1 ]  {$\overline{w}$};
	% Text Node
	\draw (291,176) node [anchor=north west][inner sep=0.75pt]  [font=\footnotesize,color={rgb, 255:red, 208; green, 2; blue, 27 }  ,opacity=1 ]  {$A$};
	% Text Node
	\draw (168,55) node [anchor=north west][inner sep=0.75pt]  [font=\footnotesize,color={rgb, 255:red, 74; green, 144; blue, 226 }  ,opacity=1 ]  {$\overline{w}$};
	% Text Node
	\draw (134,173) node [anchor=north west][inner sep=0.75pt]  [font=\scriptsize,color={rgb, 255:red, 74; green, 144; blue, 226 }  ,opacity=1 ]  {$A\ =\overline{w} -\frac{\overline{w}^{T} z_{i} \ -b}{\| z_{i} \| ^{2}_{2}}z_i$};
	% Text Node
	\draw (200,153) node [anchor=north west][inner sep=0.75pt]  [font=\tiny,color={rgb, 255:red, 74; green, 144; blue, 226 }  ,opacity=1 ]  {$\beta  >0$};
	% Text Node
	\draw (340,153) node [anchor=north west][inner sep=0.75pt]  [font=\tiny,color={rgb, 255:red, 74; green, 144; blue, 226 }  ,opacity=1 ]  {$\beta < 0$};
	\end{tikzpicture}
	\caption{Projection onto the intersection of Euclidean ball and hyperplane.}	
	\label{fig:geo}
\end{figure}
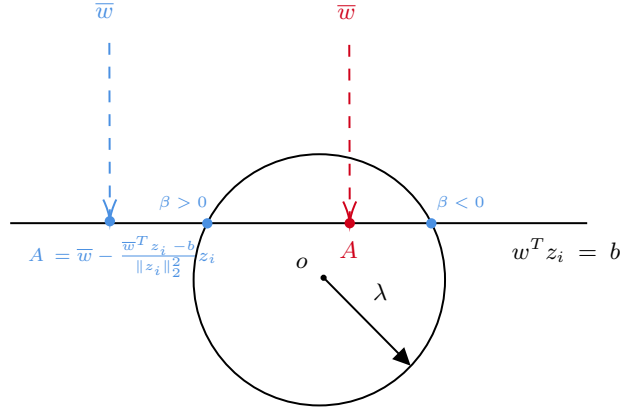	
%	\begin{figure}[ht]
%		\centering
%		\includegraphics[scale=0.6]{projectBH.png}
%		\caption{The Geometric Interpretation (Projection on the intersection of Euclidean ball and hyperplane)}
%		\label{fig:projonbh}
%	\end{figure}
\end{proof}
Let us now summarize the optimal solutions of all sub-cases and the corresponding optimality conditions in Table \ref{tb:l2}.
\begin{table}[h]	\label{tb:l2}
	\centering
	\caption{Summary of all sub-cases for $\ell_2$ proximal point update \eqref{eq:l2}} 
	\begin{tabular}{c|c|c}
		\toprule 
		Sub-cases & Optimal solution & Optimality condition\\
		\midrule  
		$\Gamma=1,j=1$ &	$ \proj_{L^d_2}(\bar{{w}}+\alpha z_i,\bar{\lambda})$ & ${w^*}^Tz < \min(\frac{\lambda^*\kappa}{2},1)$\\
		$\Gamma=1,j=2$&$ \proj_{L^d_2}(\bar{w}-\alpha z_i,\bar{\lambda}+\alpha \kappa)$  & ${w^*}^Tz>\max(\frac{\lambda^*\kappa}{2},\lambda^*\kappa-1)$\\
		$\Gamma=1,j=3$ & $ \proj_{L^d_2}(\bar{w},\bar{\lambda})$ & $1<{w^*}^Tz<\lambda^*\kappa-1$ ~\&~ $\lambda^*\kappa>2$ \\
		$\Gamma=2,(j,j')=(1,2)$ & $ \PPA(\bar{w}+\alpha z_i,\bar{\lambda},\frac{\kappa}{2},0)$ & $\lambda^*\kappa-1<{w^*}^Tz<1$ ~\&~ $\lambda^*\kappa<2$ \\
			$\Gamma=2,(j,j')=(1,3)$ & $ \PPA(\bar{w},\bar{\lambda},0,1)$ &${w^*}^Tz<\min(\frac{\lambda^*\kappa}{2},\lambda^*\kappa-1)$ ~\&~ $\lambda^*\kappa>2$ \\
			$\Gamma=2,(j,j')=(2,3)$ & $\PPA(\bar{w},\bar{\lambda},\kappa,-1)$ & ${w^*}^Tz>\max(\frac{\lambda^*\kappa}{2},1)$~\&~ $\lambda^*\kappa>2$ \\
			$\Gamma=3$ & $( \mathcal{O}_{BH}(1,\frac{2}{\kappa},\bar{w}),\frac{2}{\kappa})$ & $\lambda^*=\frac{2}{\kappa}$\\
		\bottomrule 
	\end{tabular}
\end{table}
\paragraph{Subproblems for $q=1$.} Consider the $\ell_1$ norm epigraphical projection 
\begin{equation}
\proj_{L^d_1}(x,s) = \mathop{\arg\min}_{y,t} \left\{ \frac{1}{2}\|y-x\|_2^2 + \frac{1}{2}(t-s)^2, ~\text{s.t.}~ \|y\|_1 \leq t \right\},
\label{eq:epil1} 
\end{equation}
where $L^d_1 = \{(x,s)\in \mathbb{R}^d \times \mathbb{R}: \|x\|_1 \leq s \}$.
Problem \eqref{eq:epil1} is equivalent to finding the root of an one-dimensional piecewise linear equation. By inspecting the KKT conditions (with $\lambda \ge 0$ being the Lagrangian multiplier), we have $y = \text{sign}(x)\odot \max(|x|-\lambda,0)$ (i.e., proximal operator for $\ell_1$ norm) and $t = \lambda +s$. 
Combining this with the complementary slackness condition and $\lambda>0$, the KKT conditions of \eqref{eq:epil1} reduce to the piecewise linear root-finding problem $ F(\lambda) = \sum_{i=1}^d \max(|x_i|-\lambda) - \lambda -s = 0$, which can be solved by the quick-select algorithm in linear time; see~\cite[Algorithm 2]{wang2016epigraph} for details. Otherwise, we have $\proj_{L^d_1}(x,s) = (x,s)$.

Recall that the  $\ell_1$ single-sample proximal point subproblem~\eqref{eq:l1ppa} takes the form
\begin{equation*}
\begin{aligned}
& \min_{w,\lambda,\mu}  \,\, \mu + \frac{1}{2\alpha} \left(\|w-\bar{w}\|_2^2 + (\lambda- \bar{\lambda})^2\right) \\
& \,\,\,\,\text{s.t.}~\,\,\, h_{i,j}(w,\lambda) \le \mu \,\, {\color{blue} ( \leftarrow \sigma_j \ge 0 )}, \,\, j=1,2,3; \,\,\, \|w\|_1 \leq \lambda, % \\
%& \,\,\,\,\text{s.t.}~\,\,\, h_{i,j}(w,\lambda) \le \mu, \,\, j=1,2,3; \,\,\, \|w\|_1 \leq \lambda.
\end{aligned}
\end{equation*}
where $\sigma_1,\sigma_2,\sigma_3\ge0$ are the corresponding dual multipliers. 

\begin{itemize}
	\item Case 1: $h_{i,1},h_{i,3}$ are inactive. Then, problem \eqref{eq:l1ppa} can be written as
	\[\min_{w,\lambda}\, (1+w^Tz_i-\lambda\kappa) +  \frac{1}{2\alpha}\|w-\bar{w}\|_2^2 +\frac{1}{2\alpha}(\lambda-\bar{\lambda})^2,~\ \text{s.t.} \ \|w\|_1\leq \lambda. \]
	Hence, we have $(w^*,\lambda^*) =  \proj_{L^d_1}(\bar{w} -\alpha z_i, \bar{\lambda}+\alpha \kappa)$. 
	\item Case 2: $h_{i,1}$ is active and $h_{i,3}$ is inactive. Then, problem \eqref{eq:l1-c1} can be reduced to 
	\begin{equation}\label{eq:l1-case2} 
	\min_{w,\lambda}\frac{1}{2\alpha} ( \|w-\bar{w}-\alpha z_i\|_2^2 + (\lambda -  \bar{\lambda})^2 ),~\ \text{s.t.}~w^Tz_i\leq \frac{\lambda\kappa}{2}\,{\color{blue} ( \leftarrow \sigma_1 \ge 0 )},~\|w\|_1 \leq \lambda.
	\end{equation}
	\item Cases 3 and 4:  ($h_{i,1}$ is inactive, $h_{i,3}$ is active) and ($h_{i,2}$ is inactive, $h_{i,3}$ is active). These two cases are similar to Case 2 and give rise to a problem of the form \eqref{eq:l1ppa_sub}.
\end{itemize}
Now, let us demonstrate how to solve \eqref{eq:l1ppa_sub} efficiently. Recall that 
\begin{equation*} 	
\min_{w,\lambda} \frac{1}{2\alpha} ( \|w-\bar{w}\|_2^2 + (\lambda -  \bar{\lambda})^2 ),~\ \text{s.t.}~w^Tz \le a \lambda + b \,\,{\color{blue} (\leftarrow \sigma\ge0)},\,\,\,\|w\|_1 \leq \lambda.
\end{equation*}
\begin{proposition}\label{prop:sigma2}
Suppose that $\sigma_1^*$ is the dual optimal solution to \eqref{eq:l1-case2}. Then, we have $\sigma_1^*\in[0,1]$.
\end{proposition}
\begin{proof}
	Based on the KKT conditions of \eqref{eq:l1ppa}, we have 
	\[ 1 - \sigma_1 - \sigma_2 -\sigma_3 = 0. \]
	If the optimal solution to~\eqref{eq:l1-case2} is also optimal for \eqref{eq:l1ppa}, then we can match the two KKT systems. As $h_{i,3}$ is inactive for this case, we have $\sigma_3^* = 0$. This gives
	\[  \sigma_1^* +  \sigma_2^* =1, \sigma_1^*,\sigma_2^* \ge 0 \quad\Rightarrow\quad \sigma_1^*\in[0,1].\]
\end{proof}
Proposition \ref{prop:sigma2} also holds for Cases 3 and 4. The analytic bound in Proposition~\ref{prop:sigma2} shows that $\sigma^*$ can be efficiently found by an appropriate search strategy. Next, recall from~\eqref{eq:secant} that
	\begin{align*}
	 (\hat{w}(\sigma),\hat{\lambda}(\sigma)) &  = \mathop{\arg\min}_{\|w\|_1 \leq \lambda}  \frac{1}{2\alpha} \left(\|w-\bar{w}\|_2^2 + (\lambda- \bar{\lambda})^2\right)  + \sigma (w^Tz - a\lambda - b)\\
	 & = \proj_{L^d_1}(\bar{w}-\sigma\alpha z, \bar{\lambda} +\sigma \alpha a ). 
	\end{align*} 
The following proposition establishes the monotonicity property of $\sigma \mapsto p(\sigma) = \hat{w}(\sigma)^Tz-a\kappa - b$, which plays a vital role in our development of a fast algorithm for solving~\eqref{eq:l1ppa_sub} later. 
\begin{proposition}
	If $\sigma$ satisfies (i) $\sigma = 0$ and $p(\sigma)\leq 0$, or (ii) $p(\sigma) =0$, then $(\hat{w}(\sigma),\hat{\lambda}(\sigma))$ is the optimal solution to \eqref{eq:l1ppa_sub}. Moreover, $p(\cdot)$ is continuous and monotonically non-increasing on $\mathbb{R}_+$.
\end{proposition}
\begin{proof}
	As 	$ \proj_{L^d_1}(\cdot,\cdot)$ is globally Lipschitz continuous, the function $(\hat{w}(\cdot),\hat{\lambda}(\cdot))$ is also globally
	Lipschitz continuous and further $p(\cdot)$ is continuous. Next, we prove the monotonicity property. Upon letting $h(\sigma) = \frac{1}{2\alpha} \left(\|\hat{w}(\sigma)-\bar{w}\|_2^2 + (\hat{\lambda}(\sigma)- \bar{\lambda})^2\right)$ and assuming that $0\leq \sigma_1<\sigma_2 \leq1$, we have 
	\begin{align*}
	h(\sigma_1) + \sigma_1 p(\sigma_1) & \leq h(\sigma_2) + \sigma_1 p(\sigma_2) \\& =  h(\sigma_2) + \sigma_2 p(\sigma_2) + (\sigma_1-\sigma_2) p(\sigma_2) \\
	& \leq  h(\sigma_1) + \sigma_2 p(\sigma_1) + (\sigma_1-\sigma_2) p(\sigma_2),
	\end{align*}
which implies that $p(\sigma_1) \ge p(\sigma_2)$.
\end{proof}

\begin{algorithm}[]
	\SetAlgoLined
	\caption{A modified secant algorithm to solve the RHS of \eqref{eq:l1ppa_sub}---$\text{MSA}(\bar{w},\bar{\lambda},z_i, a,b,\xi)$}
	\KwIn{Tolerance Level $\xi$\,;} 	
	\lIf{$p(0) \leq \xi$}{return $(\hat{w}(0),\hat{\lambda}(0),0)$\,}
	\Else{$\sigma_l = 0$, $r_l = -p(0)$\,; \quad \tcp{Set the lower bound $\sigma_l$ for $\sigma^*$}
		\lIf{$p(1) \ge 0$}{return $-1$\,} \tcp{$\sigma^* \in [0,1]$; see Proposition \ref{prop:sigma}}
		\Else{$\sigma_u = 1$, $r_u = -p(1)$\,; \quad\tcp{Set the upper bound  $\sigma_u$ for $\sigma^*$}}
	}
	\tcc{Secant Phase}
	$s = 1- \frac{r_l}{r_u}$, $\sigma = \sigma_u - \frac{\sigma_u-\sigma_l}{s}$; calculate $r = -p(\sigma)$\,;
	
	\While{$|r|>\xi$}{
		calculate $r = -p(\sigma)$; \quad\tcp{$\ell_1$ epigraph projection via the quick-select algorithm}
		{\uIf{$r>0$}{
				\lIf{$s\leq 2$}
				{$\sigma_u = \sigma$, $r_u= r$, $s = 1- \frac{r_l}{r_u}$, $\sigma = \sigma_u - \frac{\sigma_u-\sigma_l}{s}$
				}\Else{
					$s = \max(\frac{r_u}{r}-1,0.1)$, $\Delta\sigma= \frac{\sigma_u-\sigma}{s}$, $\sigma_u =\sigma$, $r_u = r$\,\;
					$\sigma = \max(\sigma_u-\Delta\sigma, 0.6\sigma_l+0.4\sigma_u)$, $s = \frac{\sigma_u-\sigma_l}{\sigma_u-\sigma}$\,\;
			}}\Else{
			\lIf{$s\ge 2$}
			{$\sigma_l = \sigma$, $r_l = r$, $s = 1- \frac{r_l}{r_u}$, $\sigma = \sigma_u - \frac{\sigma_u-\sigma_l}{s}$}
				\Else{$s = \max(\frac{r_l}{r}-1,0.1)$, $\Delta\sigma=\frac{\sigma-\sigma_l}{s}$, $\sigma_l =\sigma$, $r_l =r$\,\;
					$\sigma = \max(\sigma_l+\Delta\sigma, 0.6\sigma_u+0.4\sigma_l)$, $s= \frac{\sigma_u-\sigma_l}{\sigma_u-\sigma}$\,\;}}
		}
	}
	\label{algo:MSA}
\end{algorithm}

\begin{algorithm}[]
	\SetAlgoLined
	\caption{A fast algorithm based on parametric approach to solve \eqref{eq:l1ppa} }
	\KwIn{Tolerance Level $\xi$; parameters $\bar{w},\bar{\lambda}, z_i, \epsilon,\kappa$\,;}
	\tcc{Case 1: $h_{i,1},h_{i,3}$ are inactive}
	$(w^*,\lambda^*) =  \proj_{L^d_1}(\bar{w} -\alpha z_i, \bar{\lambda}+\alpha \kappa)$\,\; 	
	\lIf{$ \langle w^*,z \rangle >\max(\lambda^*\kappa-1,\frac{\lambda^*\kappa}{2})$}
	{return $(w^*,\lambda^*)$\,} \tcp{Check optimality}
	\tcc{Case 2: $h_{i,1}$ is active; $h_{i,3}$ is inactive}
	$(w^*,\lambda^*,\sigma^*) = \text{MSA}(\bar{w}+\alpha z_i,\bar{\lambda},z_i, \kappa/2,0)$\,; \tcp{Apply the modified secant algorithm \ref{algo:MSA}}
	\lIf{$ \langle w^*,z \rangle <1$ ~\&~ $\sigma^* \in [0,1]$}
	{return $(w^*,\lambda^*)$\,}
	\tcc{Case 3:$h_{i,1}$ is inactive; $h_{i,3}$ is active}
	$(w^*,\lambda^*,\sigma^*) = \text{MSA}(\bar{w},\bar{\lambda},z_i, \kappa,-1)$\,\;
	\lIf{$ \langle w^*,z \rangle >1$ ~\&~ $\sigma^* \in [0,1]$}
	{return $(w^*,\lambda^*)$\,}
	\tcc{Case 4: $h_{i,2}$ is inactive; $h_{i,3}$ is active}
	$(w^*,\lambda^*,\sigma^*) = \text{MSA}(\bar{w},\bar{\lambda},-z_i, 0,-1)$\,\;
	\lIf{$\lambda^*\kappa >2$ ~\&~ $\sigma^* \in [0,1]$}
	{return $(w^*,\lambda^*)$\,}
	\tcc{Case 5: $h_{i,1},h_{i,2},h_{i,3}$ are active}
	\Else{$w^* = \mathop{\arg\min}\limits_w\{\|w-\bar{w}\|_2^2,~\ \text{s.t.}\,\, w^Tz_i = 1, \|w\|_1 \leq \frac{2}{\kappa}\}$, $\lambda^* = \frac{\kappa}{2}$ \;
	\tcc{Apply the modified secant algorithm in \cite{dai2006new}}}
	\label{algo:l1sub}
\end{algorithm}

\subsection*{B: Convergence Rate Analysis of Incremental Algorithms}
We now give a condition under which problem \eqref{eq:drsvm_our} with $q=2$ satisfies the sharpness or quadratic growth (QG) property.  Consider the following more general formulation of the $\ell_2$-DRSVM problem:
\begin{equation}
\min_{w,\lambda} \frac{c}{2}\|w\|_2^2 +  \frac{1}{n}\sum\limits_{i=1}^n f_i(w,\lambda)+ \mathbb{I}_{ \{(w,\lambda) \in L_2^d\}}, 
\label{eq:general}
\end{equation}
where $f_1,\ldots,f_n$ are non-smooth convex functions with polyhedral epigraphs. Our condition is based the following lemma:
\begin{lemma}	\label{le:blr}
	Let $C_1, \ldots, C_N $ be closed convex subsets of $\mathbb{R}^n$ , where $C_{r+1}, \ldots, C_{N}$ are polyhedral for some $r \in \{0,1,\ldots,N\}$. Suppose that
	\[ \bigcap_{i=1}^r \ri(C_i)\cap\bigcap_{i=r+1}^NC_i \neq  \emptyset.\]
	Then, the collection $\{C_1,\ldots,C_N\}$ is boundedly linearly regular (BLR). 
\end{lemma}

\begin{proposition} Consider problem~\eqref{eq:general}. Let $\mathcal{X}$ be the set of optimal solutions and $L_2^d = \{(w,\lambda)\in\mathbb{R}^d\times\mathbb{R}:\|w\|_2 \le \lambda\}$ be the constraint set. Suppose that $\mathcal{X}\cap\ri(L_2^d) \neq \emptyset$. Then, problem~\eqref{eq:drsvm_our} satisfies the sharpness condition when $c=0$ and the QG condition when $c>0$.
\end{proposition}

\begin{proof}
	Let $x=(w,\lambda)$, $h(x) = \frac{c}{2}\|w\|_2^2+\frac{1}{n}\sum\limits_{i=1}^n f_i(w,\lambda)$, and $g(x) = h(x)+\mathbb{I}_{ \{x \in L_2^d\}}$. Consider the case where $c=0$. The set $\mathcal{X}$ can then be written as
	\[ \mathcal{X} = \{x: 0 \in \partial h(x) + \mathcal{N}_{L_2^d}(x)\},\]
	where $\mathcal{N}_{L_2^d}(x)$ is  the normal cone of $L_2^d$ at $x$.  
	As $\mathcal{X} \cap \ri(L_2^d) \neq \emptyset$, we can find an $x^* \in \mathcal{X} \cap \ri(L_2^d)$ that satisfies $0 \in \partial h(x^*)$. Thus, $x^*$ is also an optimal solution to the unconstrained problem $\min\limits_x h(x)$ and $h(x^*) = g(x^*)$. 
	
	Let $\mathcal{X}_U$ denote the set of optimal solutions to the problem $\min\limits_x h(x)$. It is not difficult to check that 
	\[ \mathcal{X} = \mathcal{X}_U \cap L_2^d.\]
	Since $f_1,\ldots,f_n$ have polyhedral epigraphs, by Lemma \ref{le:blr}, the collection $\{\mathcal{X}_U, L_2^d\}$  is BLR. This implies that there exists a constant $\kappa >0$ satisfying
	\[\dist(x, \mathcal{X}) = \dist(x, \mathcal{X}_U \cap L_2^d) \leq \kappa \dist(x, \mathcal{X}_U),~\ \forall x \in   L_2^d.\]
	Furthermore, the problem $\min_x h(x)$ enjoys the sharpness property; see~\cite[Corollary 3.6]{burke1993weak}. This gives
	\[ g(x)-g^* = h(x)-h^* \ge \sigma \dist(x,\mathcal{X}_U) \ge \frac{\sigma}{\kappa}\dist(x,\mathcal{X}),~\ \forall x\in  L_2^d. \]
	For $c>0$, we note that the problem $\min_x h(x)$ can be regarded as one with a polyhedral convex regularizer; see~\cite[Section 4.2]{zhou2017unified}. As such, it satisfies a proximal error bound (see~\cite[Proposition 6]{zhou2017unified}) and hence the QG condition (see~\cite[Theorem 4.1]{li2018calculus}). It follows that
		\[ g(x)-g^* = h(x)-h^* \ge \sigma \dist^2(x,\mathcal{X}_U) \ge \frac{\sigma}{\kappa^2}\dist^2(x,\mathcal{X}), ~\ \forall x\in  L_2^d. \]
\end{proof}
%\paragraph{Counter Example}
%A natural question is to ask whether $\ell_2$ DRSVM satisfies the sharpness property? The illustrative counter example is given as below.  
%\begin{equation}
%\begin{aligned}
%\min_{w\in\mathbb{R}^2,\lambda} |w_1-1| ~\text{s.t.}~ \ w_1^2+w_2^2 \leq \lambda^2. \\ 
%\end{aligned}
%\end{equation}
%The optimal solution is $(1,0,1)$. Thus, 
%\[ |w_1-1|-0 \ge \alpha ||(w_1,\sqrt{1-w_1^2})-(1,0)||=\alpha \sqrt{(w_1-1)^2+1-w_1^2} = \alpha'\sqrt{1-w_1}.\]
%Also, we can see the Fig.\ref{fig:counter} for the geometric interpretation. When the hyperplane tangent to the circle, the "sharpness" constant $\alpha =  \lim_{ w \to w^*}\sin(\theta) = \lim_{ w \to w^*} \frac{f(w)-f^*}{\dist(w,w^*)} = 0$. 
%
%\begin{figure}[ht]	
%	\centering
%	\includegraphics[scale=0.5]{sharp_geo.pdf}
%	\caption{Geometric Interpretation of Sharpness Property for the Counter Example }
%	\label{fig:counter}
%\end{figure}

To derive the convergence rates of the incremental algorithms, we also need the following assumption.
\begin{assumption}[Subgradient boundedness]
There exists a scalar $L>0$ such that 
\[ \|\nabla f_i(x)\| \leq L,~\ \forall~ \nabla f_i(x) \in \partial f_i(x),\ i \in [n].\]
\label{assp:l}
\end{assumption}
\vspace{-7mm}
\begin{lemma}[ISG; see~{\cite[Lemma 2.1]{nedic2001incremental}}]
	Suppose that Assumption \ref{assp:l} holds and $\{x^k=(w_0^k,\lambda_0^k)\} $ is a sequence generated by ISG. Then, for all $y$ and $k\ge 0$, we have 
	\[
	\|x^{k+1} -y\|_2^2 \leq \|x^k-y\|_2^2 - 2\alpha_kn(f(x^k)-f(y)) + a_k^2 n^2L^2. 
	\]
	\label{ISG}
\end{lemma}
\vspace{-7mm}
\begin{lemma}[IPPA]
   Suppose that Assumption \ref{assp:l} holds and $\{x^k=(w_0^k,\lambda_0^k)\} $ is a sequence generated by IPPA. Then, for all $y$ and $k\ge 0$, we have 
	\[
	\|x^{k+1} -y\|_2^2 \leq \|x^k-y\|_2^2 - 2\alpha_kn(f(x^k)-f(y)) + a_k^2 n(n+1)L^2. 
	\]
	\label{IPPA}
\end{lemma}
\vspace{-7mm}
\begin{proof}
%	\Jiajin{easy to obtain the result; Incremental proximal methods for large scale convex optimization, proposition 1; single sample recursion 
%\[ \|x_{k+1} -y\|_2^2 \leq \|x_k-y\|_2^2 - 2\alpha_k(f(x_{k+1})-f(y))\]	
% }
Based on Proposition 1 in \cite{bertsekas2011incremental} with $x^k_i = (w^k_i,\lambda^k_i)$, we have 
\begin{align*}
\|x^{k}_{i+1} -y\|_2^2 \leq \|x^k_{i}-y\|_2^2 - 2\alpha_k(f_{i+1}(x^{k}_{i+1})-f_{i+1}(y)),~\ \forall i = 0, \ldots, n-1.
\end{align*}
Summing up, 
\begin{align*}
\|x^{k}_n -y\|_2^2 &\leq  \|x^k_{0}-y\|_2^2 -  2\alpha_k\sum_{i=0}^{n-1}(f_{i+1}(x^{k}_{i+1})-f_{i+1}(y))\\
& = \|x^k_{0}-y\|_2^2 -  2\alpha_k\sum_{i=0}^{n-1}(f_{i+1}(x^{k}_{i+1})-f_{i+1}(x^{k}_{0})+f_{i+1}(x^{k}_{0})-f_{i+1}(y)) \\ 
& = \|x^k_{0}-y\|_2^2 -  2\alpha_kn(f(x^{k}_{0})-f(y))-2\alpha_k\sum_{i=0}^{n-1}(f_{i+1}(x^{k}_{i+1})-f_{i+1}(x^{k}_{0})) \\
& \leq  \|x^k_{0}-y\|_2^2 -  2\alpha_kn(f(x^{k}_{0})-f(y))+2\alpha_kL\sum_{i=0}^{n-1}\|x^{k}_{i+1}-x^{k}_{0}\|_2 \\
& \leq \|x^k_{0}-y\|_2^2 -  2\alpha_kn(f(x^{k}_{0})-f(y))+2\alpha_k^2L^2\sum_{i=0}^{n-1}(i+1) \\
& \leq \|x^k_{0}-y\|_2^2 -  2\alpha_kn(f(x^k_0)-f(y)) + a_k^2 n(n+1)L^2. 
\end{align*}
\end{proof}
Combining Lemmas \ref{ISG} and \ref{IPPA}, we have 
\[
	\|x^{k+1} -y\|_2^2 \leq \|x^k-y\|_2^2 - 2\alpha_kn(f(x^k)-f(y)) +2 a_k^2 n^2L^2. 
\]
%\begin{theorem}[Convergence Result]Let the sequence of iterates $\{x^k\}$ be generated by  ISG or IPPA. Under some mild assumptions (i.e., subgradient boundedness), we have,
%	\vspace{-2mm}
%	\begin{enumerate}[(1)] 
%		\item If the problem \eqref{eq:drsvm_our} satisfies the \textbf{sharpness} property, we apply the geometric diminishing step size scheme as $\alpha_k = \alpha_0 \rho^k$ with $\alpha_0 \ge \frac{\sigma\dist(x^{0},\mathcal{X})}{2L^2n}$and $\sqrt{1-\frac{\sigma^2}{2L^2}} \leq \rho <1$. Then, the sequence $\dist(x_k,\mathcal{X})$ converges linearly.
%		\vspace{-2mm}
%		\item If the problem \eqref{eq:drsvm_our} satisfies the \textbf{quadratic growth} property, we apply the polynomially decaying step size scheme as $\alpha_k = \frac{c}{nk}$with $c>\frac{1}{2\sigma}$. Then, the sequence $\dist(x_k,\mathcal{X})$ converges with the rate $\mathcal{O}(\frac{1}{\sqrt{k}})$ and $f(x_k)-f^*$ converges at a rate $\mathcal{O}(\frac{1}{k})$. 
%		\vspace{-2mm}
%		\item (see proposition 2.10 in ~\cite{nedic2001convergence}) For the general convex problem, the sequence $\min_{0\leq k\leq K} f(x_k) -f^*$ converges at a rate $\mathcal{O}(\frac{1}{\sqrt{K}})$ if the step size scheme $\alpha_k = \frac{c}{n\sqrt{k}}, c>0$ is adopted.
%	\end{enumerate}
%\end{theorem}
\begin{theorem} Let $\{x^k=(w_0^k,\lambda_0^k)\}$ be the sequence of iterates generated by ISG or IPPA.
	\begin{enumerate}[(1)] 
		\item If problem \eqref{eq:drsvm_our} satisfies the \emph{sharpness} condition, then by choosing the geometrically diminishing step sizes $\alpha_k = \alpha_0 \rho^k$ with $\alpha_0 \ge \tfrac{\sigma\dist(x^{0},\mathcal{X})}{2L^2n}$ and $\sqrt{1-\tfrac{\sigma^2}{2L^2}} \leq \rho <1$, the sequence $\{x^k\}$ converges linearly to an optimal solution to~\eqref{eq:drsvm_our}; i.e., $\dist(x^k,\mathcal{X})\leq \mathcal{O}(\rho^k)$ for all $k\ge0$.
		\vspace{-2mm}
		\item If problem \eqref{eq:drsvm_our} satisfies the \emph{quadratic growth} condition, then by choosing the polynomially decaying step sizes $\alpha_k = \tfrac{\gamma}{nk}$ with $\gamma>\frac{1}{2\sigma}$, the sequence $\{x^k\}$ converges to an optimal solution to~\eqref{eq:drsvm_our} at the rate $\mathcal{O}(\tfrac{1}{\sqrt{k}})$ and $\{f(x^k)-f^*\}$ converges to zero at the rate $\mathcal{O}(\tfrac{1}{k})$. 
		\vspace{-2mm}
		\item (See~\cite[Proposition 2.10]{nedic2001convergence}) For the general convex problem~\eqref{eq:drsvm_our}, by choosing the step sizes $\alpha_k = \tfrac{\gamma}{n\sqrt{k}}$ with $\gamma>0$, the sequence $\displaystyle\{\min_{0\leq k\leq K} f(x^k) -f^*\}$ converges to zero at the rate $\mathcal{O}(\tfrac{1}{\sqrt{K}})$.
	\end{enumerate}
\end{theorem}
\begin{proof}
(1):\\
By the sharpness condition $f(x_k)-f^* \ge \sigma \dist(x_k,\mathcal{X})$, we have
	\[
	\dist^2(x^{k+1},\mathcal{X}) \leq \dist^2(x^{k},\mathcal{X}) - 2\alpha_k \sigma n \dist(x^{k},\mathcal{X}) +2\alpha_k^2L^2n^2.
	\]
	We now prove by induction that
	\[ \dist(x^{k},\mathcal{X}) \leq \frac{2\alpha_0L^2n}{\sigma}\rho^k.\]
	The base case trivially holds, as $\dist(x^{0},\mathcal{X}) \leq \frac{2\alpha_0L^2n}{\sigma}$. For the inductive step, we compute
	\begin{equation}
	\begin{aligned}
	\dist^2(x^{k+1},\mathcal{X}) & \leq \left( \frac{2\alpha_0L^2n}{\sigma}\rho^k \right)^2 - 2\alpha_k \sigma n \frac{2\alpha_0L^2n}{\sigma}\rho^k+2\alpha_k^2L^2n^2 \\ 
	& =\frac{4\alpha_0^2L^4n^2}{\sigma^2}\rho^{2k}- 2\alpha_0^2L^2n^2\rho^{2k}\\
	& = \frac{4\alpha_0^2L^4n^2}{\sigma^2}\rho^{2k} \left( 1-\frac{\sigma^2}{2L^2} \right) \\
	& \leq \left( \frac{2\alpha_0L^2n}{\sigma} \right)^2\rho^{2(k+1)}.
	\end{aligned}
	\end{equation}
	This completes the proof.

	(2):
	By the quadratic growth condition $f(x_k)-f^* \ge \sigma \dist^2(x_k,\mathcal{X})$, we have 
	\[
	\dist^2(x^{k+1},\mathcal{X}) \leq (1-2\alpha_k \sigma n)\dist^2(x^{k},\mathcal{X}) +2\alpha_k^2L^2n^2.
	\]
	Plugging in the corresponding step size scheme $\alpha_k = \frac{\gamma}{nk}$, we obtain
		\[
	\dist^2(x^{k+1},\mathcal{X}) \leq (1-\frac{2\gamma\sigma}{k})\dist^2(x^{k},\mathcal{X}) + \frac{2\gamma^2L^2}{k^2}.
	\]
	We now prove by induction that
	\[ \dist^2(x^{k},\mathcal{X}) \leq \frac{B}{k}, \]
	where $B>0$ is a given number. Indeed, we have
	\begin{align*}
	\dist^2(x^{k+1},\mathcal{X}) & \leq \left( 1-\frac{2\gamma\sigma}{k} \right) \frac{B}{k}+\frac{2\gamma^2L^2}{k^2} \\ 
	 & = \frac{B}{k+1}+\frac{B}{k(k+1)} - \frac{2\gamma\sigma B}{k^2}+\frac{2\gamma^2L^2}{k^2}\\
	 & \leq  \frac{B}{k+1} + \frac{B}{k^2} - \frac{2\gamma\sigma B}{k^2}+\frac{2\gamma^2L^2}{k^2} \\	 
	 & = \frac{B}{k+1} + \frac{(1-2\gamma\sigma)B}{k^2}+\frac{2\gamma^2L^2}{k^2} \\
	 & \leq \frac{B}{k+1}, 
	\end{align*}
	where the last inequality holds if $ \frac{(1-2\gamma\sigma)B}{k^2}+\frac{2\gamma^2L^2}{k^2} <0$. Hence, we have $B > \frac{2\gamma^2L^2}{2\gamma\sigma-1}$ due to $\gamma>\frac{1}{2\sigma}$. Combining this with the base case, we have $B >  \max\{ \frac{2\gamma^2L^2}{2\gamma\sigma-1}, \dist^2(x^{0},\mathcal{X}) \}$. 
\end{proof}

%\newpage 
\subsection*{C: Additional Experimental Results}
To begin with, we show how to extend the GS-ADMM framework in \cite{li2019first} to tackle our DRSVM problems, which serves as a baseline in Section \ref{sec:exp}. We reformulate problem \eqref{eq:drsvm_our} as 
\begin{equation}\label{eq:drsvm_s}
\begin{aligned}
& \underset{w,s,\lambda}{\text{min}}
& & \lambda\epsilon + \frac{1}{n} \sum_{i=1}^n s_i \\
&\,\,\, \text{s.t.}
& &  \ 1 - w^Tz_i \le s_i, i\in[n], \\
&&&  \ 1 + w^Tz_i + \lambda \kappa \le s_i , i\in[n], \\
&&&  \ s_i \ge 0, i\in[n], \\
&&& \ \|w\|_q \leq \lambda.
\end{aligned} 
\end{equation}
We follow the technique used in ~\cite[Proposition 3.1]{li2019first}.
\begin{proposition}
	Suppose that $(w^*,\lambda^*,s^*)$ is a global minimizer of \eqref{eq:drsvm_s}. Then, we have $\lambda^* \le \lambda^U = \tfrac{1}{\epsilon}$. 
\end{proposition}
\begin{proof}
For simplicity, we consider the case where $q=2$. Since problem \eqref{eq:drsvm_s} satisfies the Managasarian-Fromovitz Constraint Qualification (MFCQ), the KKT conditions are necessary and sufficient. Let $a_{ij} \ge 0, \forall j\in[3], i\in[N]$ and $\beta \ge 0$ be the dual variables. Then, we can write down the KKT conditions as follows:
	\begin{equation}
	\label{LKKT}
	% \text{Stationary Condition} \longrightarrow 
	\left\{ 
	\begin{aligned}
	& \sum_{i=1}^N (a_{i2}-a_{i1})z_i+2\beta w = 0,\\
	& a_{i1}+a_{i2}+a_{i3} = \frac{1}{N}, \forall i \leq N,\\
	& \kappa\sum\limits_{i=1}^N a_{i2} +2\beta\lambda= \epsilon,\\
	& a_{i1}(1 - w^Tz_i - s_i) = 0, \\
	& a_{i2}(1 + w^Tz_i-\lambda\kappa - s_i) = 0, \\
	& a_{i3}s_i = 0, \\
	& {\color{blue}{\beta(||w||_2^2-\lambda^2) = 0}}. 
	\end{aligned}
	\right.
	\end{equation}
	Based on \eqref{LKKT}, we have 
	\[
	\begin{aligned}
	\mathbf{0} & = \sum\limits_{i=1}^N (a_{i2}-a_{i1})w^Tz_i + 2\beta||w||_2^2 \\ 
	&  =\sum\limits_{i=1}^N (a_{i2}-a_{i1})w^Tz_i + 2\beta\lambda^2 
	= \sum\limits_{i=1}^N(a_{i2}-a_{i1})w^Tz_i+\lambda(\epsilon-\kappa\sum\limits_{i=1}^N a_{i2})\\
	& = \lambda\epsilon+\sum\limits_{i=1}^Na_{i2}(w^Tz_i-\lambda\kappa)-\sum\limits_{i=1}^Na_{i1}w^Tz_i  = \lambda\epsilon+\sum\limits_{i=1}^N (a_{i2}+a_{i1})(s_i-1).
	\end{aligned}
	\]
	Thus, we have 
	\[
	\lambda =  \frac{1}{\epsilon}\sum\limits_{i=1}^N (a_{i2}+a_{i1})(1-s_i) = \frac{1}{\epsilon}\sum\limits_{i=1}^N (\frac{1}{N}-a_{i3})(1-s_i) \le \frac{1}{\epsilon N}\sum\limits_{i=1}^N (1-s_i) \le \frac{1}{\epsilon}.
	\]
\end{proof}

\begin{Remark}
Although we focus on the case where $q=2$ in this proof, we can easily extend the techniques to study the case where $q\in\{1,\infty\}$. We just need to modify the blue part in~\eqref{LKKT}. Specifically, observe that $\|w\|_1 \leq \lambda$ is equivalent to $Bw\leq \lambda e_{2^d}$, where $B$ is the $2^d \times d $ matrix whose rows are all the possible arrangements of $+1$'s and $-1$'s. On the other hand, $\|w\|_\infty \leq \lambda$ is equivalent to $e_i^Tw \leq \lambda$, $-e_i^Tw \leq \lambda$, $\forall i \in [n]$. 
\end{Remark}
	
 Subsequently, we develop a standard ADMM algorithm to address the $w$-subproblem
\[
\begin{aligned}
& \min_{w}\frac{1}{n}\sum\limits_{i=1}^n \max \left\{1-w^Tz_i, 1+w^Tz_i-\lambda\kappa,0\right\},~\text{s.t.} \ \|w\|_q \leq \lambda. 
\end{aligned}
\]
We apply the operator splitting technique to reformulate it as 
\[
\begin{aligned}
& \underset{w,y}{\text{min}}
& &\frac{1}{n}\sum\limits_{i=1}^n \max \left\{1-y_i, 1+y_i-\lambda\kappa,0\right\} \\
&\,\, \text{s.t.}
& & Zw-y=0,\;\\
&&& \|w\|_q \leq \lambda.
\end{aligned} 
\]
\begin{algorithm}[!htbp]\label{algo:ADMM}
	\SetAlgoLined
	\caption{ADMM for solving $w$-subproblem}
	\KwIn{Choose value $(w^0,y^0,g^0) \in \mathbb{R}^d \times \mathbb{R}^n \times \mathbb{R}^n$; \\
		\quad  \quad \quad Initialized the penalty parameter $\rho_0$ and shrinking parameter $\gamma \geq 1$;} 
	\KwOut{$\{(w^k,y^k,g^k)\}_{k=1}^K$ and function value sequences; }
	\For{each iteration }{
	\tcc{Accelerated projected gradient algorithm, see \cite[Algorithm 5]{li2019first}}
	$w^{k+1} = \arg\min\limits_{\|w\|_q \leq \lambda} \left\{\frac{\rho_k}{2}\|Zw-y^k+\frac{g^k}{\rho_k}\|_2^2\right\}$\;
	\tcc{Closed-form update}
	$y^{k+1} =\arg\min\limits_{y \in \mathbb{R}^n} \left\{ \frac{1}{n}\sum\limits_{i=1}^n \max \left\{1-y_i, 1+y_i-\lambda\kappa,0\right\} + \frac{\rho_k}{2}\|y-Zw^{k+1} -\frac{g^k}{\rho_k}\|_2^2\right\} $\;
	\tcc{Dual update}
	$ g^{k+1} = g^{k}+\rho_k(Zw^{k+1}-y^{k+1})$\;
	$\rho_{k+1} = \gamma\rho_{k}$\;

}
\end{algorithm}
\paragraph{Single-sample proximal point update for $c>0$} Recall that 
\begin{equation*}
\begin{aligned}
& \underset{w,\lambda}{\text{min}}
& & \  \frac{c}{2}\|w\|_2^2 + \max\left\{1-w^Tz_i, 1+w^Tz_i-\lambda\kappa,0\right\} + \frac{1}{2\alpha}(\|w-\bar{w}\|_2^2 + (\lambda-\bar{\lambda})^2) \\
& \,\,\text{s.t.}
& & \ \|w\|_q \leq \lambda.
\end{aligned} 
\end{equation*}
Note that $\mu = \frac{\lambda}{\sqrt{1+\alpha c}}$ and $\kappa'  =\kappa\sqrt{1+\alpha c}$. The above problem can be written as
\begin{equation}\label{eq:csub}
\begin{aligned}
& \underset{w,\mu}{\text{min}}
& & \   \max\left\{1-w^Tz_i, 1+w^Tz_i-\mu\kappa',0\right\} + \frac{1+ac}{2\alpha}\left(\left\|w-\frac{\bar{w}}{1+ac}\right\|_2^2 + (\mu-\bar{\mu})^2\right) \\
& \,\,\text{s.t.}
& & \ \|w\|_q \leq \sqrt{1+\alpha c} \mu. 
\end{aligned} 
\end{equation}
Indeed, problem \eqref{eq:csub} shares the same structure as problem \eqref{eq:drsvm_ppa}.
%\begin{table}[H]
%	\caption{Wall-clock Time Comparison on UCI  Real Dataset: $\ell_1$-DRSVM, $c=1,\kappa=1,\epsilon=0.1$}
%	\begin{tabular}{@{}ccccccccc@{}}
%		\toprule
%		Dataset & \multicolumn{4}{c}{\begin{tabular}[c]{@{}c@{}}Objective Value\\ M-ISG | IPPA | Hybrid | YALMIP\end{tabular}} & \multicolumn{4}{c}{\begin{tabular}[c]{@{}c@{}}Wall-clock time (sec)\\ M-ISG | IPPA | Hybrid | YALMIP\end{tabular}} \\ \midrule
%		a1a & 0.7871456 & 0.7871450 & { \textbf{{0.7871445}}} & 0.7871462& {\textbf{{1.233}}} & 2.169& 1.545 & 16.495 \\
%		a2a & 0.8002882 & \textbf{0.8002879} & { \textbf{{0.8002879}}} & 0.8003101 & \textbf{0.602} & 3.546 & {\text{{0.720}}} & 23.557 \\
%		a3a & 0.7826654 & \textbf{0.7826653} & \textbf{0.7826653} & \textbf{0.7826653} & \textbf{3.687} & 4.626 & 3.696 & 32.575 \\
%		a4a & \textbf{0.7929724} & \textbf{0.7929724} & \textbf{0.7929724} & 0.7929959 & \textbf{0.579} & 4.109 & 0.713 & 62.456 \\
%		a5a & 0.7852845 & 0.7852848  & {\textbf{{0.7852844}}} & 0.7853440 & \textbf{0.639} & 2.929 & {\text{{1.156}}} & 109.620 \\
%		a6a &\textbf{0.7764425} &\textbf{0.7764425} & \textbf{0.7764425} & 0.7767701 & \textbf{1.272} & 5.756 & 1.576 & 185.080 \\
%		a7a & 0.7822116 & \textbf{0.7822114} & 0.7822115 & 0.7827171 & \textbf{2.121} & 4.336 & 2.976 & 270.480 \\
%		a8a & \textbf{0.7805498} & \textbf{0.7805498} & \textbf{0.7805498} & 0.7836023
%		 & \textbf{2.502} & 10.184 & 2.798 & 372.050 \\
%		a9a & 0.7767114 & \textbf{0.7767099} & 0.7767113 & 0.7791881
%		 & \textbf{3.018} & 9.295 & 6.203 & 642.160 \\ \bottomrule
%	\end{tabular}
%\end{table}

\begin{table}[H]
	\centering
	\caption{Wall-clock Time Comparison on UCI  Real Dataset: $\ell_1$-DRSVM, $c=1,\kappa=1,\epsilon=0.1$}
	\setlength{\tabcolsep}{4.5pt}
	\begin{tabular}{@{}ccccclcccc@{}}
		\toprule
		{\color{white}{\multirow{2}{*}{Dataset}}} & \multicolumn{4}{c}{Objective Value} &  & \multicolumn{4}{c}{Wall-clock time (sec)} \\ \cmidrule(l){2-10} 
		& M-ISG & IPPA & Hybrid & YALMIP &  & M-ISG & IPPA & Hybrid & YALMIP \\ \midrule
		a1a & 0.7871456 & 0.7871450 & \textbf{0.7871445} & 0.7871462 &  & \textbf{1.233} & 2.169 & 1.545 & 16.495 \\
		a2a & 0.8002882 & \textbf{0.8002879} & \textbf{0.8002879} & 0.8003101 &  & \textbf{0.602} & 3.546 & 0.720 & 23.557 \\
		a3a & 0.7826654 & \textbf{0.7826653} & \textbf{0.7826653} & \textbf{0.7826653} &  & \textbf{3.687} & 4.626 & 3.696 & 32.575 \\
		a4a & \textbf{0.7929724} & \textbf{0.7929724} & \textbf{0.7929724} & 0.7929959 &  & \textbf{0.579} & 4.109 & 0.713 & 62.456 \\
		a5a & 0.7852845 & 0.7852848 & \textbf{0.7852844} & 0.7853440 &  & \textbf{0.639} & 2.929 & 1.156 & 109.620 \\
		a6a & \textbf{0.7764425} & \textbf{0.7764425} & \textbf{0.7764425} & 0.7767701 &  & \textbf{1.272} & 5.756 & 1.576 & 185.080 \\
		a7a & 0.7822116 & \textbf{0.7822114} & 0.7822115 & 0.7827171 &  & \textbf{2.121} & 4.336 & 2.976 & 270.480 \\
		a8a & \textbf{0.7805498} & \textbf{0.7805498} & \textbf{0.7805498} & 0.7836023 &  & \textbf{2.502} & 10.184 & 2.798 & 372.050 \\
		a9a & 0.7767114 & \textbf{0.7767099} & 0.7767113 & 0.7791881 &  & \textbf{3.018} & 9.295 & 6.203 & 642.160 \\ \bottomrule
	\end{tabular}
\end{table}

\begin{table}[H]
	\centering
	\caption{Wall-clock Time Comparison on UCI  Real Dataset: $\ell_\infty$-DRSVM, $c=1,\kappa=1,\epsilon=0.1$}
		\setlength{\tabcolsep}{4.5pt}
	\begin{tabular}{@{}ccccclcccc@{}}
		\toprule
		{\color{white}{\multirow{2}{*}{Dataset}}} & \multicolumn{4}{c}{Objective Value} &  & \multicolumn{4}{c}{Wall-clock time (sec)} \\ \cmidrule(l){2-10} 
		& M-ISG & IPPA & Hybrid & YALMIP &  & M-ISG & IPPA & Hybrid & YALMIP \\ \midrule
		a1a & 0.7853266 & \textbf{0.7853265} & \textbf{0.7853265} & 0.7853269 &  & \textbf{0.510} & 0.707 & 0.687 & 12.928 \\
		a2a & 0.7987669 & 0.7987666 & 0.7987667 & \textbf{0.7987663} &  & \textbf{0.861} & 1.233 & 1.182 & 20.850 \\
		a3a & 0.7810149 & \textbf{0.7810140} & 0.7810145 & 0.7810568 &  & \textbf{0.350} & 0.666 & 0.563 & 28.494 \\
		a4a & \textbf{0.7913534} & \textbf{0.7913534} & \textbf{0.7913534} & 0.7913963 &  & \textbf{0.540} & 1.100 & 0.679 & 63.037 \\
		a5a & 0.7836246 & \textbf{0.7836189} & 0.7836246 & 0.7836478 &  & \textbf{0.737} & 1.421 & 0.897 & 96.314 \\
		a6a & 0.7748542 & \textbf{0.7748533} & 0.7748537 & 0.7759606 &  & \textbf{1.203} & 2.250 & 1.763 & 201.510 \\
		a7a & \textbf{0.7806894} & \textbf{0.7806894} & \textbf{0.7806894} & 0.7811091 &  & \textbf{1.753} & 3.244 & 2.124 & 370.510 \\
		a8a & 0.7789801 & \textbf{0.7789800} & \textbf{0.7789800} & 0.7850996 &  & \textbf{2.390} & 4.568 & 2.981 & 365.140 \\
		a9a & \textbf{0.7750633} & \textbf{0.7750633} & \textbf{0.7750633} & 0.7776798 &  & \textbf{3.368} & 6.706 & 4.371 & 753.330 \\ \bottomrule
	\end{tabular}
\end{table}

\begin{Remark}
For $\ell_2$-DRSVM problems, it is worth mentioning that the $c>0$ case is identical to the $c=0$ case from a modeling perspective, which depends on the different robustness levels $\epsilon$ and $\kappa$. 

 \end{Remark}
\end{document}